\documentclass[11pt, leqno]{amsart}
\usepackage{amsmath,amssymb,txfonts}
\usepackage{amssymb}
\usepackage{amsxtra}
\usepackage{amsthm, color}
\usepackage{txfonts}
\usepackage{graphicx}
\usepackage{times}
\usepackage{citeref}
\usepackage{tikz}
\usepackage{pgfplots}
\usepackage{tikz-3dplot}
\numberwithin{equation}{section}
\usepackage{color}

\usepackage[cp1252]{inputenc}
\usepackage{bbm}
\usepackage{mathrsfs}
\usepackage{graphicx}

\usepackage[active]{srcltx}

\usepackage{graphicx}

\newtheorem{prop}{Proposition}[section]
\newtheorem{theorem}[prop]{Theorem}
\newtheorem{lemma}[prop]{Lemma}

\usepackage{tikz}
\usetikzlibrary{calc}



\begin{document}

	\title[\tiny
	A Unified Quermassintegral Approach to Quasilinear Heat Dispersion and Loss]{A Unified Quermassintegral Approach to \\ Quasilinear Heat Dispersion and Loss}
	
	\author{\tiny Xiaoshang Jin}
	\address{School of Mathematics and Statistics, Huazhong University of Science and Technology, Wuhan, Hubei 430074, China}
	\email{jinxs@hust.edu.cn}
	\author{\tiny Jie Xiao}
	\address{Department of Mathematics \& Statistics,
		Memorial University, St. John's, NL A1C 5S7, Canada}
	\email{jxiao@math.mun.ca}


	\begin{abstract}
		This paper establishes a fundamental connection between quasilinear potential theory and convex geometric analysis by investigating the interplay between the
		quasilinear Laplace operator and quermassintegrals. We introduce a quasilinear heat dispersion law for convex conductors and prove that, among all convex conductors of a fixed mean width, the closed ball is a unique maximizer of this dispersion. By characterizing the
		quasilinear heat loss of a convex conductor explicitly in terms of its quermassintegrals, we demonstrate not only a formal equivalence between the isocapacitary and isoperimetric inequalities in the setting of mathematical physics but also that, among all convex conductors of a fixed mean width, the closed ball is a unique maximizer of this loss. These results provide a novel bridge between the metric properties of convex conductors and the variational analysis of quasilinear elliptic operators, offering a unified perspective on sharp geometric inequalities and their extremal cases.
	\end{abstract}

	\thanks{The first-named \& second-named authors were supported by {the Fundamental Research Funds for the Central Universities HUST: \#2025BRSXB002} \& NSERC of Canada \# 202979 respectively.}
	
	\date{}

%    General info

%\date{}

%\dedicatory{In memory of Adriano M. Garsia 1928-2010}

%\dedicatory{Dedicated to Adriano M. Garsia who surely appreciated a simplified approach}


\maketitle

\tableofcontents

\section{Introduction}\label{s1}
\setcounter{equation}{0}

\subsection{Brief of quasilinear heat dispersion/loss}\label{s11}

Given not only a real pair $(p,{\lambda})\in [1,\infty)\times(0,\infty)$ but also a compact connected subset $K$ of the closure $\Omega$ of an open set $\Omega$ in the Euclidean space $\mathbb R^{n\ge 2}$ which not only exists as a conductor of the unit temperature but also is thermally insulated by surrounding it with the layer of thermal insulator $\Omega\setminus K$ with Lipschitz bounary pair $\{\partial K,\partial\Omega\}$, the quasilinear heat dispersion of $(K,\Omega)$ is determined by such a mathematical physic quantity (cf. \cite{AC, DNT})
\begin{equation}
\label{e11}
{{{\mathsf{H^d}}}}_{p,{\lambda}}(K,\Omega)=\underset{f\in W^{1,p}(\Omega)\ \&\ f\big|_K=1}{\inf}\left(\int_{\Omega}|\nabla f|^p\,d\mathcal{L}^n+{\lambda}\int_{\partial\Omega}|f|^p\,d\mathcal{H}^{n-1}\right),
\end{equation}
that refers to the quasilinear spreading of heat from a concentrated source over a larger area or volume, where
$$
\begin{cases}
W^{1,p}(\Omega)=\text{the first-order Sobolev $p$-space on $\Omega$};\\
d\mathcal{L}^n=\text{the $n$-dimensional Lebesgue measure element};\\
d\mathcal{H}^{n-1}=\text{the $(n-1)$-dimensional Hausdorff measure element}.
\end{cases}
$$

Below are two special circumstances.
\begin{itemize}
	\item Since
	$$
	f\in W^{1,p}_0(\Omega)=\overline{C_0^\infty(\Omega)}^{W^{1,p}(\Omega)} \Longrightarrow f\big|_{\partial\Omega}=0\ \ \text{by}\ \ C_0^\infty(\Omega)=\big\{\text{smooth functions compactly supported in $\Omega$}\big\},
	$$
	there holds
\begin{equation}\label{e12}
 {{{\mathsf{H^d}}}}_{p,{\lambda}}(K,\Omega)\le {{\mathsf{H^l}}}_p(K,\Omega)=\underset{f\in W^{1,p}_0(\Omega)\ \&\ f\big|_K=1}{\inf}\int_{\Omega}|\nabla f|^p\,d\mathcal{L}^n
\end{equation} where ${{\mathsf{H^l}}}_p(K,\Omega)$ is not only $(K,\Omega)$'s quasilinear heat loss (or $p$-capacitance) -i.e.- the quasilinear transfer of thermal energy away from a system (e.g., a building, hot object) to its cooler surroundings but also has such a geometric endpoint (cf. \cite[Lemma 2.2.5]{MazBook}{
	$$
	{{\mathsf{H^l}}}_1(K,\Omega)=\inf\Bigg\{\mathcal{H}^{n-1}(\partial O):\ K\subset O\subset\overline{O}\subset \Omega\ \forall\ \hbox{open}\ O\ \hbox{with smooth}\ \partial O\Bigg\}.
	$$
}
	Interestingly, \cite[Theorem 1.1]{LiHo} (cf. \cite{BMP} for $p=2$) reveals
	\begin{equation}
	\label{e13}
	{{\mathsf{H^l}}}_p(K,\Omega)= {{\mathsf{H^l}}}_{\Delta_p}(K,\Omega)=\underset{f\in C^\infty_0(\Omega)\ \&\ f\big|_K=1}{\inf}2^{-1}\int_{\Omega}|\Delta_p f|\,d\mathcal{L}^n\ \ \text{under}\ \ p\in (1,\infty),
	\end{equation}
     where
     $$
     \begin{cases} \Delta_p f=\nabla\cdot(|\nabla f|^{p-2}\nabla f)=\text{the quasilinear Laplace operator}\\
     \big(\text{whose {\color{red}{red arrow lengths}} induce quasilinear diffusions}\big):\\
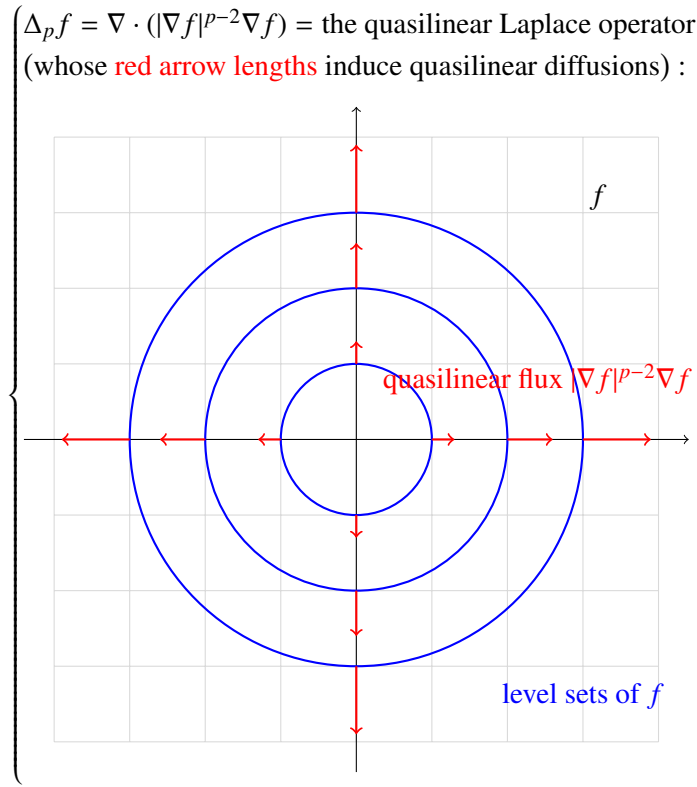
     \begin{tikzpicture}[scale=2]

     % Surface grid
     \draw[step=0.5,gray!30,thin] (-2,-2) grid (2,2);

     % Axes
     \draw[->] (-2.2,0) -- (2.2,0) node[right] {${}$};
     \draw[->] (0,-2.2) -- (0,2.2) node[above] {${}$};

     % Level sets (solution f)
     \draw[thick,blue] (0,0) circle (0.5);
     \draw[thick,blue] (0,0) circle (1);
     \draw[thick,blue] (0,0) circle (1.5);

     % Gradient vectors
     \foreach \r in {0.5,1,1.5}{
     	\draw[->,red,thick] (\r,0) -- ({\r+0.3*\r},0);
     	\draw[->,red,thick] (-\r,0) -- ({-\r-0.3*\r},0);
     	\draw[->,red,thick] (0,\r) -- (0,{\r+0.3*\r});
     	\draw[->,red,thick] (0,-\r) -- (0,{-\r-0.3*\r});
     }

     % Labels
     \node at (1.6,1.6) {$f$};
     \node[red] at (1.2,0.4) {quasilinear flux $|\nabla f|^{p-2}\nabla f$};
     \node[blue] at (1.5,-1.7) {level sets of $f$};

     \end{tikzpicture}
	\end{cases}
	$$
	\item If $f=1$ in $\Omega,$ then
	one has
	\begin{equation}
	\label{e14}
	{{{\mathsf{H^d}}}}_{p,{\lambda}}(K,\Omega)\le {\lambda}\mathcal{H}^{n-1}(\partial\Omega)\ \ \forall\ \ p\in [1,\infty).
	\end{equation}
	Moreover, if the interior $K^\circ$ of $K$ is equal to $\Omega$, then
	$$K=\overline{\Omega}\ \ \&\ \ |\nabla f|\big|_\Omega=0=(f-1)\big|_{\partial\Omega},
	$$
	whence \eqref{e14} reaches its equality
	$$
	{{{\mathsf{H^d}}}}_{p,{\lambda}}(K,\Omega)={\lambda}\mathcal{H}^{n-1}(\partial\Omega)\ \ \forall\ \ p\in [1,\infty).
	$$
	Even more interesting, \cite[Theorem 1.2]{LiHo} (cf. \cite{BP} for $p=2$) indicates
	\begin{equation}
	\label{e15}
	\mathcal{H}^{n-1}(\partial K)
	={{\mathsf{H^l}}}_{\Delta_p}(\partial K,\partial\Omega)=
	\underset{f\in C_0^1(\overline{\Omega})\, \&\, \Delta_p f\in C(\overline{\Omega})\, \&\, -|\nabla f|^{p-2}\frac{\partial f}{\partial\nu}\big|_{\partial K}\ge 1}{\inf}\int_{\Omega}|\Delta_p f|\,d\mathcal{L}^n.
	\end{equation}
	
\end{itemize}

\subsection{Statement of Theorem \ref{t11}}\label{s12}

 Needless to say, a careful look at \eqref{e12}-\eqref{e13}-\eqref{e14}-\eqref{e15} tells us that an intrinsic study relies truly on the geometric structure of the conductor $(K,\Omega)$. If $K$ is convex compact subset of $\mathbb R^n$ with its interior $K^\circ\not=\emptyset$ (denoted simply by $K\in\mathscr{K}^{n}$), $\mathbb B^n$ is the unit open ball of $\mathbb R^n$ with compact boundary $\mathbb S^{n-1}$ \& closure $\overline{\mathbb B^n}$, and
$$
K+t{\mathbb B}^n=\big\{x+t y: (x,y)\in K\times{\mathbb B}^n\big\}\ \ \forall\ \  t\in (0,\infty),
$$
then (cf. \cite[p.213, (4.8)]{Sch}))
\begin{equation}
\label{e17}
\begin{cases}
\mathcal{H}^{n-1}\big(\partial(K+t{\mathbb B}^n)\big)=n\sum\limits_{j=0}^{n-1} \binom{n-1}{j}{\mathcal W}_{j+1}(K)t^j=\mathcal{H}^{n-1}(\partial K)+\cdots+\sigma_{n-1} t^{n-1};\\
\mathcal{L}^n(K+t{\mathbb B}^n)=\sum\limits_{j=0}^n\binom{n}{j}{\mathcal W}_j(K)t^j=\mathcal{L}^n(K)+n\mathcal{W}_1(K)t+\cdots+\upsilon_n t^n;\\
\mathcal{W}_j(K)=\text{the so-called $\{0,1,...,n\}\ni j$-quermassintegral of $K$};\\
\mathcal{W}_0(K)=\text{the Lebesgue volume of $K$};\\
n\mathcal{W}_1(K)=\mathcal{H}^{n-1}(\partial K)=\text{the surface area of $K$};\\
\mathcal{W}_{n-1}(K)=\text{the mean width of $K$};\\
\mathcal{L}^n\big(\overline{\mathbb B^n}\big)=\upsilon_n=n^{-1}\sigma_{n-1}=n^{-1}\mathcal{H}^{n-1}(\mathbb S^{n-1}).
\end{cases}
\end{equation}
Especially, if $\partial K$ is of class $C^2$ with its Gaussian map being a diffeomorphism (cf. \cite[p.296, (5.55)]{Sch}) then for $j\in\{1,...,n\}$ there is
$$
\begin{cases}
\mathcal{W}_j(K)=n^{-1}\int_{\partial K} \bigg(\frac{\varsigma_{j-1}\big(\kappa_1,...,\kappa_{n-1};\partial K\big)}{\binom{n-1}{j-1}}\bigg)\,d\mathcal{H}^{n-1};\\
\varsigma_{j-1}\big(\kappa_1,...,\kappa_{n-1};\partial K\big)=\sum\limits_{i_1<i_2<\cdots<i_{j-1}}\kappa_{i_1}\kappa_{i_2}\cdots\kappa_{i_{j-1}};\\
\kappa_j=\text{the $j$-th principal curvature on $\partial K$}.
\end{cases}
$$
Thus, the geometric quantity pair
$$
\Big\{{{\mathsf{H^l}}}_j(\cdot,\cdot), \mathcal{W}_{j}(\cdot)\Big\}
$$
has the same homogeneous degree
$$n-j\ \ \text{as}\ \  j\in\{1,...,n-1\}.
$$
This, along with not only \cite{DNT, Ba, AC, ChW} but also \eqref{e11} \& \eqref{e13}, intrinsically motivates us to achieve the following new quasilinear heat dispersion law through quermassintegrals.

\begin{theorem}
	\label{t11} Let
	$$\begin{cases}
	(K,\Omega)=\text {a conductor in $\mathbb R^n$ {with Lipschitz boundary pair $\{\partial K,\partial\Omega\}$}};\\
	(p-1,{\lambda})\in (1,\infty)^2;\\
	{{{\mathsf{H^d}}}}_{\Delta_p,{\lambda}}(K,\Omega)= \underset{\tiny
		f\in C^2(\overline{\Omega})\, \&\, f\big|_K=1\, \&\,
		({\frac{\partial f}{\partial\nu}}){|\nabla f|^{p-2}}+{{\lambda} f}{|f|^{p-2}}=0\ \text{on}\, \partial\Omega}{\inf} 2^{-1}\left(\int_{\Omega}|\Delta_p f|\,d\mathcal{L}^{n}+{\lambda}\int_{\partial\Omega}|f|^{p-1}{d\mathcal{H}^{n-1}}\right).
		\end{cases}
		$$
		\begin{itemize}
			\item[\rm(i)] There is the identification
		\begin{equation}{\color{red}
		\label{e18}
		{{{\mathsf{H^d}}}}_{\Delta_p,{\lambda}}(K,\Omega)={{{\mathsf{H^d}}}}_{p,{\lambda}}(K,\Omega).}
		\end{equation}
			
		\item [\rm (ii)]	
		If $K^\ast$ is a closed ball of the same mean width as $\mathscr{K}^n\ni K$'s one - i.e. -
\begin{equation}
\label{eMW}
\mathcal{W}_{n-1}(K)=\mathcal{W}_{n-1}\big(K^\ast\big),
\end{equation}
then there is
		\begin{equation}{\color{red}
		\label{e19}
		{{{\mathsf{H^d}}}}_{\Delta_p,{\lambda}}\big(K,K+t{\mathbb B}^n\big)\le {{{\mathsf{H^d}}}}_{\Delta_p,{\lambda}}\big(K^\ast,K^\ast+t{\mathbb B}^n\big) \ \ \forall\ \  t\in (0,\infty),}
		\end{equation}
		with equality iff $K$ is a closed ball.
		
		\end{itemize}
\end{theorem}

\subsection{Statement of Theorem \ref{t12}}\label{s13}
For any compact subset $K$ of $\mathbb{R}^n$, let
\begin{equation}\label{e111}
{{\mathsf{H^l}}}_p(K)=\lim\limits_{\Omega\rightarrow\mathbb{R}^n}{{\mathsf{H^l}}}_p(K,\Omega)=\text{the $p$-capacitance of $K$}.
\end{equation}
The limit in \eqref{e111} not only exists but also does not depend on the choice of $\Omega.$ By the P\'olya description within \cite{Po}, up to a constant factor, ${{\mathsf{H^l}}}_2(K)$ = the electric capacity of $K$ can be interpreted as the (linear) heat lost by $K$ in unit time within steady state while this heat propagates through $\mathbb R^n$. Thus, ${{\mathsf{H^l}}}_p(K)$ is naturally regarded as a total heat loss at the quasilinear equilibrium -i.e.- the quasilinear heat lost by $K$ through space $\mathbb R^n$ in unit time within steady state. Moreover, there is the following
variational form just like \eqref{e11}:
$$
{{\mathsf{H^l}}}_p(K)=\underset{f\in W^{1,p}_0(\mathbb{R}^n)\ \&\ f\big|_K=1}{\inf}\left(\int_{\mathbb{R}^n}|\nabla f|^p\,d\mathcal{L}^n\right).
$$

As an extremal circumstance of Theorem \ref{t11}, we surpringly discover the coming-up-next brand-new law for the quasilinear heat loss through quermassintegrals.

\begin{theorem}\label{t12}
 Let $1<p<n$.
 \begin{itemize}

 	\item[\rm (i)] There is the equivalence
 	\begin{equation}{\color{red}
 	\label{e417e}
 	{\rm (IC)}\ \frac{\left(\frac{{{\mathsf{H^l}}}_{p}(K)}{{{\mathsf{H^l}}}_{p}(\overline{\mathbb{B}^n})}\right)^\frac{1}{n-p}}{\left(\frac{\mathcal{W}_0(K)}{\upsilon_n}\right)^\frac{1}{n}}\ge 1
 	\ \ \forall\ \ K\in\mathscr{K}^n\Longleftrightarrow {\rm (IP)}\ \frac{\left(\frac{\mathcal{W}_1(K)}{\upsilon_{n}}\right)^\frac1{n-1}}{\left(\frac{\mathcal{W}_0(K)}{\upsilon_n}\right)^\frac{1}{n}}\ge 1\ \ \forall\ \  K\in
 	\mathscr{K}^n.}
 	\end{equation}
 	Naturally, {\rm(IC)} \& {\rm(IP)} are always true with their equalities holding iff $K$ is a closed ball.	
 	
		\item[\rm (ii)] For $K\in\mathscr{K}^{n\ge 3}$ with {Lipschitz} boundary $\partial K$ there is
	\begin{equation}
	\label{e113}
	{{\mathsf{H^l}}}_{p}(K)\le\begin{cases} n\left(\frac{n-p}{p-1}\right)^{p-1}\mathcal{W}_1\left(\Big(\frac{\mathcal{W}_2(K)}{\mathcal{W}_1(K)}\Big)^\frac{p-1}{n-1}K\right)=n\left(\frac{n-p}{p-1}\right)^{p-1}\mathcal{W}_2\left(\Big(\frac{\mathcal{W}_2(K)}{\mathcal{W}_1(K)}\Big)^\frac{p-2}{n-2}K\right);\\
	\sigma_{n-1}\left(\frac{n-p}{p-1}\right)^{p-1}\Big(\upsilon_n^{-1}\mathcal{W}_{n-1}(K)\Big)^{n-p},
	\end{cases}
	\end{equation}
	with its inequality becoming an equality iff $K$ is a closed ball. Especially, either \eqref{e113}'s first inequality under $p=2=n-1$ or \eqref{e113}'s second inequality under $1<p<n\ge 3$ amounts to that if $K^\ast$ is a closed ball with the mean width condition \eqref{eMW} then
	\begin{equation}{\color{red}
	\label{eMC}
	{\mathsf{H^l}}_p(K)\le{\mathsf{H^l}}_p(K^\ast),}
	\end{equation}
	with equality iff $K$ is a closed ball.
	\end{itemize}
\end{theorem}
\bigskip

\subsection{Structure of the paper}\label{s14}
Based on such a limiting process that if
$(K,\Omega)$ is a smooth conductor in $\mathbb R^n,$ then
$$
\begin{cases}
{{{\mathsf{H^d}}}}_{p,\infty}(K,\Omega)=\underset{{\lambda}\to\infty}{\lim}{{{\mathsf{H^d}}}}_{p,{\lambda}}(K,\Omega)=\underset{f\in C^1_0(\overline{\Omega})\ \&\ f\big|_{K}=1}{\inf}\left(\int_{\Omega}|\nabla f|^p\,d\mathcal{L}^{n}\right);\\

{{{\mathsf{H^d}}}}_{\Delta_p,\infty}(K,\Omega)=\underset{{\lambda}\to\infty}{\lim}{{{\mathsf{H^d}}}}_{\Delta_p,{\lambda}}(K,\Omega)=\underset{
	f\in C^2_0(\overline{\Omega}),\ \& \ f\big|_K=1}{\inf}2^{-1}\left(\int_{\Omega}|\Delta_p f|\,d\mathcal{L}^{n}\right),
\end{cases}
$$
whose physical interpretation ensures a very natural change from heat dispersion to heat loss (cf. \eqref{e13}):
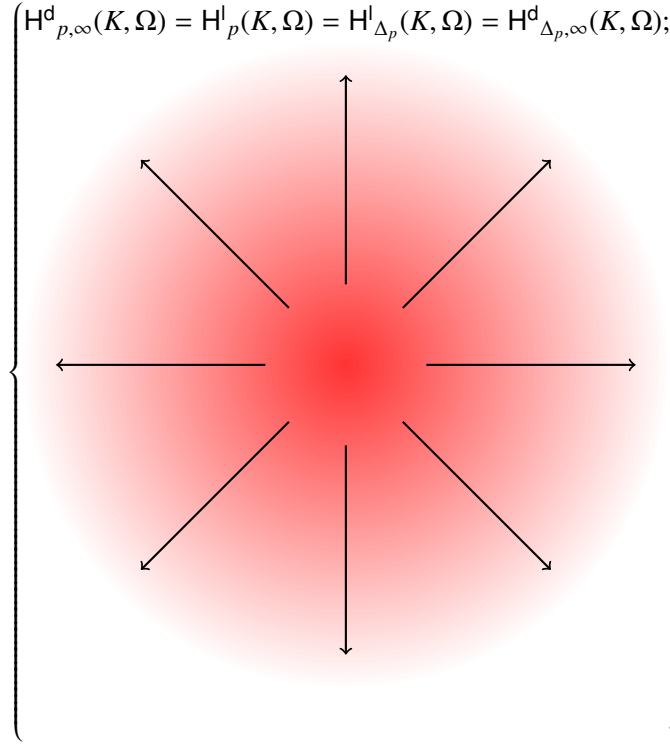
$$
\begin{cases}
{{{\mathsf{H^d}}}}_{p,\infty}(K,\Omega)={{\mathsf{H^l}}}_p(K,\Omega)={{\mathsf{H^l}}}_{\Delta_p}(K,\Omega)={{{\mathsf{H^d}}}}_{\Delta_p,\infty}(K,\Omega);\\
\begin{tikzpicture}[scale=2.13]

% Shaded heat distribution
\shade[inner color=red!80, outer color=white] (0,0) circle (2);

% Heat flow arrows
\foreach \angle in {0,45,...,315}
{
	\draw[->, thick] ({0.5*cos(\angle)}, {0.5*sin(\angle)})
	-- ({1.8*cos(\angle)}, {1.8*sin(\angle)});
}

\node at (0,-2.2) {};

\end{tikzpicture}.
\end{cases}
$$
we will not only prove but also deepen Theorems \ref{t11}-\ref{t12} in \S\ref{s3}-\S\ref{s4} respectively.

Last but not least, in complement to Theorems \ref{t11}-\ref{t12}, the final Appendix A is enclosed to especially resolve the minimizing problem
$$
\underset{K\in \mathscr{K}^n\ \&\ \mathcal{W}_j( K)\ge {\mathcal W}_j(\overline{\mathbb B^n})}{\inf}\frac{\left(\frac{\mathcal{W}_j(K)}{\mathcal{W}_j(\overline{\mathbb B^n})}\right)^\frac1{n-j}}{\left(\frac{\int_K h\,d\mathcal{L}^n}{\|h\|_{L^1(\mathbb R^n)}}\right)^\frac{1}{n}}\ \ \forall\ \ j\in\{0,1,...,n-1\}\ \ \&\ \ h\ge 0\  \ \text{with}\ \ \|h\|_{L^1(\mathbb R^n)}>0,
$$
through certain essentials of the quermassintegrals within \eqref{e17}.

\section{Demonstration for Theorem \ref{t11}}\label{s3}
\setcounter{equation}{0}

\subsection{Proof of Theorem \ref{t11}} This consists of three parts.

\subsubsection*{Validity of (i)} Thanks to the basic fact that any $K\in\mathscr{K}^n$ can be approximated in the Hausdorff distance $\hbox{dist}_H$ by $\mathscr{K}^n$'s members with smooth boundaries, it is enough to validate \eqref{e18} for any smooth conductor $(K,\Omega)$ in $\mathbb R^n$ according to the following three phases which optimize the argument for \cite[Theorem 1.1]{JX1} over smooth compact manifolds.

\begin{itemize}
 \item First of all, it is plain to check
	$$
	\begin{cases}
	{{{\mathsf{H^d}}}}_{p,{\lambda}}(K,\Omega) =\underset{f\in \mathfrak{C}^{0,1}(K,\Omega)}{\inf}\left(\int_{\Omega\setminus K}|\nabla f|^p\,d\mathcal{L}^{n}+{\lambda}\int_{\partial\Omega}|f|^{p}\,d\mathcal{H}^{n-1}\right);\\
	{{{\mathsf{H^d}}}}_{\Delta_p,{\lambda}}(K,\Omega) =\underset{f\in \mathfrak{C}^{1,1}_{{\lambda},p}(K,\Omega)}{\inf}2^{-1}\left(\int_{\Omega\setminus K}|\Delta_p f|\,d\mathcal{L}^{n}+{\lambda}\int_{\partial\Omega}|f|^{p-1}\,d\mathcal{H}^{n-1}\right),
	\end{cases}
	$$
	where
	$$
	\begin{cases} \mathfrak{C}^{0,1}(K,\Omega)=\Big\{f\in C^{0,1}(\overline{\Omega\setminus K}):\ f|_{\partial K}=1\Big\};\\
	\mathfrak{C}^{1,1}_{{\lambda},p}(K,\Omega)=\Bigg\{f\in C^{1,1}(\overline{\Omega\setminus K}):\  f\big|_{\partial K}-1=|\nabla f|\big|_{\partial K}=\bigg({|\nabla f|^{p-2}}{(\frac{\partial f}{\partial\nu})}+{{\lambda} f}{|f|^{p-2}}\bigg)\Bigg|_{\partial\Omega}=0\Bigg\}.
	\end{cases}
	$$

\item Next, without loss of generality, we may assume that each function $f$ in the above-defined classes satisfies that $0\leq f\leq 1$ - otherwise - we can consider $-f$ or $2-f$ instead. In the sequel, we prove \eqref{e18} according to two circumstances.
	\begin{itemize}
		\item On the one hand, if
		$$
		f\in \mathfrak{C}^{1,1}_{{\lambda},p}(K,\Omega)\subseteq\mathfrak{C}^{0,1}(K,\Omega),
		$$
		then
		$$
		\begin{aligned}
		\int_{\Omega\setminus K}|\Delta_p f|\,d\mathcal{L}^{n} &\geq \int_{\Omega\setminus K}(1-2f) \Delta_p f\,d\mathcal{L}^{n}
		\\ &=\int_{\partial(\Omega\setminus K)}\left(\frac{(1-2f)\frac{\partial f}{\partial\nu}}{|\nabla f|^{2-p}}\right)\,d\mathcal{H}^{n-1}-\int_{\Omega\setminus K}\left(\frac{\big(\nabla(1-2f)\big)\cdot\nabla f}{|\nabla f|^{2-p}}\right)\,d\mathcal{L}^{n}
		\\ &=\int_{\partial \Omega}(1-2f)\Big(-{\lambda}|f|^{p-2}f\Big)\,d\mathcal{H}^{n-1}+2\int_{\Omega\setminus K}|\nabla f|^p\,d\mathcal{L}^{n}
		\\ &=2\Bigg(\int_{\Omega\setminus K}|\nabla f|^p\,d\mathcal{L}^{n}+{\lambda}\int_{\partial \Omega} |f|^p\,d\mathcal{H}^{n-1}\Bigg)-{\lambda}\int_{\partial \Omega} |f|^{p-1}\,d\mathcal{H}^{n-1}.
		\end{aligned}$$
		This in turn derives
		\begin{equation}
		\label{e31}
		{{{\mathsf{H^d}}}}_{\Delta_p,{\lambda}}(K,\Omega) \geq {{{\mathsf{H^d}}}}_{p,{\lambda}}(K,\Omega).
		\end{equation}
		\item On the other hand, thanks to $p-1,{\lambda}\in (0,\infty)$, there is always a minimizer $u_\dagger$ for ${{{\mathsf{H^d}}}}_{p,{\lambda}}$ such that (cf. \cite[(1.2)]{AC} \& \cite[(3.2)]{Ba})
		\begin{equation}
		\label{e32}
		\begin{cases}
		u_\dagger=1&\ \ \text{in}\ \ K;\\
		\Delta_p u_\dagger=0&\ \ \text{in}\ \ \Omega\setminus K;\\
		|\nabla u_\dagger|^{p-2}\frac{\partial u_\dagger}{\partial\nu}+{\lambda}|u_\dagger|^{p-2}u_\dagger=0&\ \ \text{on}\ \ \partial\Omega\setminus \partial K,
		\end{cases}
		\end{equation}
		where \eqref{e32}'s third equation is understood under the weak sense that
for any $\phi\in W^{1,p}(\Omega)$ and $\phi\big|_K=0,$
		\begin{equation}
		\label{e33}
		\int_{\Omega}|\nabla u_\dagger|^{p-2}\nabla u_\dagger\cdot\nabla \phi\,d\mathcal{L}^n+{\lambda}\int_{\partial\Omega}|u|^{p-2}u\phi\,d\mathcal{H}^{n-1}=0.
		\end{equation}
		
	Moreover, for a suitablly small number $\delta>0$ let's choose
		\begin{equation}\label{e34}
		\begin{cases}
		f\in C^2[0,1];\\
        f'\geq 0;\\
		 f\big|_{[0,1-\delta]}=1;\\
		 f'(1-\delta)=f''(1-\delta)=f(1)=f'(1)=0;
        \\ w=h(u_\dagger)=1-(1-u_\dagger)f(u_\dagger).
		 \end{cases}
		 \end{equation}
		Then
		$$
		\begin{cases}
		\dot{h}(u_\dagger)=h'(u_\dagger)=tf'(t)\big|_{t=u_\dagger}\geq 0;\\
		\Delta_pw=\dot{h}^{p-1}\Delta_p u_\dagger+(\dot{h}^{p-1})'|\nabla u_\dagger|^p=(\dot{h}^{p-1})'|\nabla u_\dagger|^p\ \ \text{in}\ \ \Omega\setminus K.
		\end{cases}
		$$
		Upon noticing that
		$$
		w=u_\dagger\ \ \text{when}\ \ u_\dagger\leq 1-\delta,
		$$
		we have
		\begin{equation}\label{e35}
		\int_{\Omega\setminus K}|\Delta_p w|\,d\mathcal{L}^{n} = \int_{\{1-\delta<u_\dagger<1\}}  |\Delta_p w|\,d\mathcal{L}^{n}
		=\int_{1-\delta}^1\big|(\dot{h}^{p-1})'(t)\big|\left(\int_{\{u_\dagger=t\}}|\nabla u_\dagger|^{p-1}\,d\mathcal{H}^{n-1}\right)\,dt,
		\end{equation}
		whence making a two-fold treatment.
		\begin{itemize}
		
		\item On the one hand, letting $t$ be sufficiently close to $1$, along with \eqref{e32}'s second \& third conditions, ensures
		\begin{equation}\label{e36}
	\int_{\{u_\dagger=t\}}|\nabla u_\dagger|^{p-1}\,d\mathcal{H}^{n-1} =\int_{\Omega\cap\{u<t\}}\Delta_p u_\dagger\,d\mathcal{L}^{n}-\int_{\partial\Omega}|\nabla u|^{p-2}\left(\frac{\partial u_\dagger}{\partial\nu}\right)\,d\mathcal{H}^{n-1}
		={\lambda}\int_{\partial\Omega}  |u_\dagger|^{p-2}u_\dagger d\mathcal{H}^{n-1}.
		\end{equation}
		\item On the other hand, we can select a special function $f$ such that
		\begin{equation}
		\label{e37}
		\int_{1-\delta}^1\big|(\dot{h}^{p-1})'(t)\big|\,dt\to 1\ \ \text{as}\ \ \delta\to 0.
		\end{equation}
		As a matter of fact, given a sufficiently small $\varepsilon>0,$ let
		$$
		\begin{cases} \delta=\frac{2\varepsilon^2\pi}{1-2\varepsilon}+\varepsilon\pi;\\
		h''(t)=\begin{cases}
		0&\forall\ \ t\in \Big[0,1- \frac{2\varepsilon^2\pi}{1-2\varepsilon}-\varepsilon\pi\Big); \\
		\sin\frac{1}{\varepsilon}\Big(t-1+\frac{2\varepsilon^2\pi}{1-2\varepsilon}+\varepsilon\pi\Big)&\forall\ \ t\in \Big[1- \frac{2\varepsilon^2\pi}{1-2\varepsilon}-\varepsilon\pi,1-\frac{2\varepsilon^2\pi}{1-2\varepsilon}\Big);\\ \left(\frac{1+2\varepsilon}{2}\right)\left(\frac{2\varepsilon-1}{2\varepsilon^2}\right)\sin\frac{1-2\varepsilon}{2\varepsilon^2}
		\Big(t-1+\frac{2\varepsilon^2\pi}{1-2\varepsilon}\Big) &\forall\ \ t\in \Big[1-\frac{2\varepsilon^2\pi}{1-2\varepsilon},1\Big].
		\end{cases}
		\end{cases}
		$$
		Then
		$$\begin{aligned}
		h'(t)&=1+\int_0^t h''(s)ds
		\\&=\begin{cases}
		1&\forall\ \  t\in \Big[0,1- \frac{2\varepsilon^2\pi}{1-2\varepsilon}-\varepsilon\pi\Big); \\
		1+\varepsilon\Big(1-\cos\big(\frac{t-1}{\varepsilon}+\frac{2\varepsilon\pi}{1-2\varepsilon}+\pi\big)\Big)&\forall\ \ t\in \Big[1- \frac{2\varepsilon^2\pi}{1-2\varepsilon}-\varepsilon\pi,1-\frac{2\varepsilon^2\pi}{1-2\varepsilon}\Big);\\
		1+2\varepsilon+\left(\frac{1+2\varepsilon}{2}\right)\Bigg(-1+\cos\Big(
		\frac{t-1+\frac{2\varepsilon^2\pi}{1-2\varepsilon}}{\big(\frac{1-2\varepsilon}{2\varepsilon^2}\big)^{-1}}\Big)\Bigg)& \forall\ \ t\in\Big[1-\frac{2\varepsilon^2\pi}{1-2\varepsilon},1\Big].
		\end{cases}
		\end{aligned}
		$$
		Consequently, we read off that
			\begin{equation}\label{e38}
		\begin{cases} [0,1]\ni t\mapsto h(t)=\int_0^t h'(s)ds\ \ \text{is in $C^2[0,1]$};\\
		h(t)\big|_{[0,1-\delta]}=t;\\
		h'(1-\delta)=1;\\
		h''(1-\delta)=0;\\
		h'(1)=0=h''(1);\\
		h(1)=\int_0^{1-\frac{2\varepsilon^2\pi}{1-2\varepsilon}-\varepsilon\pi} 1 dt+\int_{1-\frac{2\varepsilon^2\pi}{1+\varepsilon}-\varepsilon\pi}^{1-\frac{2\varepsilon^2\pi}{1-2\varepsilon}}(1+\varepsilon)dt+
		\int_{1-\frac{2\varepsilon^2\pi}{1-2\varepsilon}}^1\Big(\frac{1+2\varepsilon}{2}\Big)dt=1.
		\end{cases}
\end{equation}
		Obviously, \eqref{e38} is equivalent to \eqref{e34}. Accordingly, there holds the required limiting process \eqref{e37}:
		$$\begin{aligned}
		\int_{1-\delta}^1\big|(\dot{h}^{p-1})'(t)\big|\,dt &=(\dot{h}^{p-1})(t)\big|^{1-\frac{2\varepsilon^2\pi}{1-2\varepsilon}}_{1- \frac{2\varepsilon^2\pi}{1-2\varepsilon}-\varepsilon\pi}+(\dot{h}^{p-1})(t)\big|^{1-\frac{2\varepsilon^2\pi}{1-2\varepsilon}}_{1}
		\\ &=2(1+2\varepsilon)^{p-1}-1
		\\ &\rightarrow 1 \ \ \text{as}\ \ \varepsilon \rightarrow 0\ \ \text{or}\ \ \delta\to 0.
		\end{aligned}
		$$
		\end{itemize}
		
		Now, for a sufficiently small $\varepsilon>0,$ we have
		$$
		\begin{cases} w\in\mathfrak{C}^{1,1}_{{\lambda},p}(K,\Omega);\\ w|_{\partial\Omega}=u_\dagger|_{\partial\Omega},
		\end{cases}
		$$
		thereby utilizing \eqref{e35}-\eqref{e36}-\eqref{e37} to deduce
		$$\begin{aligned}
		2{{{\mathsf{H^d}}}}_{\Delta_p,{\lambda}}(K,\Omega) &\leq \int_{\Omega\setminus K}|\Delta_p w|\,d\mathcal{L}^{n}+{\lambda}\int_{\partial\Omega}|w|^{p-2}w\,d\mathcal{H}^{n-1}
		\\ & =2(1+2\varepsilon)^{p-1}{\lambda}\int_{\partial\Omega}|u_\dagger|^{p-2}u_\dagger\,d\mathcal{H}^{n-1}
		\\ & =2(1+2\varepsilon)^{p-2}{{{\mathsf{H^d}}}}_{p,{\lambda}}(K,\Omega)
		\end{aligned}$$
		where the last equality (cf. \cite[(3.3)]{Ba}) follows from using not only the test function
		$\phi=u_\dagger-1$ within \eqref{e33} but also $u_\dagger$'s minimizing property
		$$
		{{{\mathsf{H^d}}}}_{p,{\lambda}}(K,\Omega)=\int_{\Omega}|\nabla u_\dagger|^p\,d\mathcal{L}^n+{\lambda}\int_{\partial\Omega}|u_\dagger|^p\,d\mathcal{H}^{n-1}.
		$$
		Upon sending $\varepsilon\rightarrow 0$ in the last estimation, we obtain
		\begin{equation}
		\label{e39}
		{{{\mathsf{H^d}}}}_{\Delta_p,{\lambda}}(K,\Omega)  \leq {{{\mathsf{H^d}}}}_{p,{\lambda}}(K,\Omega).
		\end{equation}
		\end{itemize}
	
\item Finally, a combination of \eqref{e31}\&\eqref{e39} derives \eqref{e18} for any smooth conductor $(K,\Omega)$ in $\mathbb R^n$, thereby verifying \eqref{e18}'s equality case.

\end{itemize}

\subsubsection*{Validity of (ii)} Due to the just-proven law \eqref{e18}, we are only required to prove \eqref{e19} for ${{{\mathsf{H^d}}}}_{p,{\lambda}}$. However, upon noting that \eqref{e19}'s setting for ${{{\mathsf{H^d}}}}_{p,{\lambda}}$ amounts to \cite[Theorem 5.1]{Ba} which generalizes \cite[Theorem 4.1]{DNT}, we are about to debilitate the strongest hypothesis \eqref{eMW} through
reformulating \cite[Theorem 5.1]{Ba}'s proof in the sequel.
\begin{itemize}
	\item Firstly,
	since $K^\ast$ within the last equation is a closed ball, we may not only choose a nonnegative radial minimizer $u_{\dagger,\ast}$ for $${{{\mathsf{H^d}}}}_{p,{\lambda}}(K^\ast, K^\ast+t\mathbb B^n)\ \ \text{connecting to \eqref{e32} for $u_\dagger$}
	$$ but also define
	$$
	\begin{cases}
	r_\ast=\text{the radius of $K^\ast$};\\
	s_\ast=u_{\dagger,\ast}(r_\ast+t)=\underset{K^\ast+t\mathbb B^n}{\min}u_{\dagger,\ast};\\
	u_{\dagger,\ast}\big|_{\overline{\mathbb B^n_{r_\ast}}}=1.
	\end{cases}
	$$
	Upon noticing that $u_{\dagger,\ast}$ is radial, we have
	$$
	|\nabla u_{\dagger,\ast}|=\text{constant on the level sets of $u_{\dagger,\ast}$}.
	$$
	\item Secondly, via defining
	$$
	\begin{cases}
	g(s)=|\nabla u_{\dagger,\ast}|\big|_{\{u_{\dagger,\ast}=s\}} &\forall\ \ s\in (s_\ast, 1];\\
	{\rm d}(x)=\underset{y\in K}{\inf}|x-y|=\text{dist}(x,K)&\forall\ \ x\in K+t\mathbb B^n;\\
	w(x)=G\big(r_\ast+{\rm d}(x)\big)&\forall\ \ x\in K+t\mathbb B^n;\\
	G^{-1}(s)=r_\ast+\int_{s}^1\big(g(\lambda)\big)^{-1}\,d\lambda&\forall\ \ s\in (s_\ast, 1],
	\end{cases}
	$$
	we get that not only $G$ is a decreasing function but also
	$$
	\begin{cases}
	w\in W^{1,p}(K+t\mathbb B^n);\\
	\underset{K+t\mathbb B^n}{\max}w=w\big|_{K}=G(r_\ast)=1;\\
	\underset{K+t\mathbb B^n}{\min}{w}=w\big|_{\partial(K+t\mathbb B^n)}=G(r_\ast+t);\\
	|\nabla w|\big|_{\{w=s\}}=|\nabla u_{\dagger,\ast}|\big|_{\{u_{\dagger,\ast}=s\}}=g(s)\ \ \forall\ \ s\in \Big[ \underset{K+t\mathbb B^n}{\min}{w},1\Big].
	\end{cases}
	$$
	Accordingly, there holds
	\begin{equation}
	\label{e311}
	{{{\mathsf{H^d}}}}_{p,{\lambda}}(K, K+t\mathbb B^n)\le \int_{(K+t\mathbb B^n)\setminus K}|\nabla w|^p+{\lambda}\int_{\partial(K+t\mathbb B^n)}|w|^p\,d\mathcal{H}^{n-1}.
	\end{equation}
	
	\item Thirdly, upon not only letting
	$$
	\begin{cases}
	E_s=\big\{x\in K+t\mathbb B^n:\ w(x)>s\big\}=\big\{x\in K+t\mathbb B^n:\ {\rm d}(x)<G^{-1}(s)\big\}=K+G^{-1}(s)\mathbb B^n;\\
	F_s=\big\{x\in K^\ast+t\mathbb B^n:\ u_{\dagger,\ast}(x)>s\big\}=K^\ast+G^{-1}(s)\mathbb B^n,
	\end{cases}
	$$
	but also using (cf. \eqref{e21})
		\begin{equation}
		\label{blue1}
		\begin{cases}
		\mathcal{W}_{n-1}(K)=\mathcal{W}_{n-1}(K^\ast);\\
		\left(\frac{\mathcal{W}_j(K^\ast)}{\mathcal{W}_j(\overline{\mathbb B^n})}\right)^\frac1{n-j}=\frac{\mathcal{W}_{n-1}(K^\ast)}{\mathcal{W}_{n-1}(\overline{\mathbb B^n})}=\frac{\mathcal{W}_{n-1}(K)}{\mathcal{W}_{n-1}(\overline{\mathbb B^n})}\ge \left(\frac{\mathcal{W}_{j}(K)}{\mathcal{W}_{j}(\overline{\mathbb B^n})}\right)^\frac1{n-j}&\ \ \forall\ \ j\in\{0,...,n-1\};\\
		\mathcal{W}_j(K)\le\mathcal{W}_{j}(K^\ast)&\ \ \forall\ \ j\in\{0,...,n\},
		\end{cases}
		\end{equation}
	we obtain that if
	$$
	\underset{K+t\mathbb B^n}{\min}{w}<s\le 1\ \ \&\ \ \rho=G^{-1}(s),
	$$
	then \eqref{e17}'s first formula gives {
		\begin{equation}
		\label{blue2}
		\mathcal{H}^{n-1}(\partial E_s)=n\sum_{k=0}^{n-1} \binom{n-1}{k}{\mathcal W}_{k+1}(K)\rho^k\le n\sum_{k=0}^{n-1} \binom{n-1}{k}{\mathcal W}_{k+1}(K^\ast)\rho^k=\mathcal{H}^{n-1}(\partial F_s).
		\end{equation}}
	This \eqref{blue2} in turn implies that if $\underset{K+t\mathbb B^n}{\min}{w}<s\le 1$ then
	\begin{align*}
	\int_{\{w=s\}}\frac{d\mathcal{H}^{n-1}}{|\nabla w|^{1-p}}=\big(g(s)\big)^{p-1}\mathcal{H}^{n-1}(\partial E_s)\le \big(g(s)\big)^{p-1}\mathcal{H}^{n-1}(\partial F_s)=
	\int_{\{u_{\dagger,\ast}=s\}}\frac{d\mathcal{H}^{n-1}}{|\nabla u_{\dagger,\ast}|^{1-p}}.
	\end{align*}
	\item Fourthly, an application of the co-area formula derives {
		\begin{equation}\label{blue3}
		\int_{(K+t\mathbb B^n)\setminus K}\frac{d\mathcal{L}^{n}}{|\nabla w|^{-p}}=\int_{\underset{K+t\mathbb B^n}{\min}{w}}^1\frac{\mathcal{H}^{n-1}(\partial E_s)\,ds}{\big(g(s)\big)^{1-p}}
		\le \int_{\underset{K+t\mathbb B^n}{\min}{w}}^1\frac{\mathcal{H}^{n-1}(\partial F_s)\,ds}{\big(g(s)\big)^{1-p}}\le
		\int_{(K^\ast+t\mathbb B^n)\setminus K^\ast}\frac{d\mathcal{L}^{n}}{|\nabla u_{\dagger,\ast}|^{-p}}.
		\end{equation}}
	Meanwhile, upon employing (cf. \eqref{blue2})
	$$
	\begin{cases}
	w|_{\partial(K+t\mathbb B^n)}=\underset{K+t\mathbb B^n}{\min}{w}=\underset{K+t\mathbb B^n}{\min}{u_{\dagger,\ast}};\\
	\mathcal{H}^{n-1}(\partial(K+t\mathbb B^n))\le \mathcal{H}^{n-1}(\partial(K^\ast+t\mathbb B^n)),
	\end{cases}
	$$
	we achieve {
		\begin{equation}\label{blue4}
		\int_{\partial(K+t\mathbb B^n)}\frac{d\mathcal{H}^{n-1}}{|w|^{-p}}=\frac{
			\Big(\underset{K+t\mathbb B^n}{\min}{|w|}\Big)^{p}}{\big(\mathcal{H}^{n-1}(\partial(K+t\mathbb B^n))\big)^{-1}}
		\le\frac{\Big(\underset{K+t\mathbb B^n}{\min}{|u_{\dagger,\ast}}|\Big)^{p}}{\big(\mathcal{H}^{n-1}(\partial(K^\ast+t\mathbb B^n))\big)^{-1}}
		=\int_{\partial(K^\ast+t\mathbb B^n)}\frac{d\mathcal{H}^{n-1}}{|u_{\dagger,\ast}|^{-p}}.
		\end{equation}}
	
	\item Fifthly, a combination of \eqref{e311} \& \eqref{blue3}-\eqref{blue4} derives \eqref{e19}'s first inequality:
	\begin{align*}
	{{{\mathsf{H^d}}}}_{p,{\lambda}}\big(K,K+t{\mathbb B}^n\big)&\le
	\int_{(K+t\mathbb B^n)\setminus K}|\nabla w|^p+{\lambda}\int_{\partial(K+t\mathbb B^n)}|w|^p\,d\mathcal{H}^{n-1}\\
	&\le \int_{(K^\ast+t\mathbb B^n)\setminus K^\ast}|\nabla u_{\dagger,\ast}|^p\,d\mathcal{L}^{n}+{\lambda}\int_{\partial(K^\ast+t\mathbb B^n)}|u_{\dagger,\ast}|^p\,d\mathcal{H}^{n-1}\\
	&={{{\mathsf{H^d}}}}_{p,{\lambda}}\big(K^\ast,K^\ast+t{\mathbb B}^n\big).
	\end{align*}
	\end{itemize}
	
	Now, a careful examination of the foregoing \eqref{blue1}-\eqref{blue2}-\eqref{blue3}-\eqref{blue4}, along with \cite[Theorem 4.1]{Ba} for $p>1=n-1$ \& \cite[Theorem 3.1]{DNT} for $p=2=n$, reveals that {if
		there exists the following surface area restriction (weaker than \eqref{eMW})
		$$
		\mathcal{W}_1(K+t\mathbb B^n)=n\mathcal{H}^{n-1}(\partial(K+t\mathbb B^n))\le n\mathcal{H}^{n-1}(\partial(K^\ast+t\mathbb B^n))=\mathcal{W}_1(K^\ast+t\mathbb B^n)\ \ \forall\ \ t\in[0,\lambda],
		$$
		then \eqref{blue3}-\eqref{blue4} hold, thereby implying
		\begin{align*}
		{{{\mathsf{H^d}}}}_{\Delta_p,{\lambda}}\big(K,K+t{\mathbb B}^n\big)={{{\mathsf{H^d}}}}_{p,{\lambda}}\big(K,K+t{\mathbb B}^n\big)\le {{{\mathsf{H^d}}}}_{p,{\lambda}}\big(K^\ast,K^\ast+t{\mathbb B}^n\big)={{{\mathsf{H^d}}}}_{\Delta_p,{\lambda}}\big(K^\ast,K^\ast+t{\mathbb B}^n\big),
		\end{align*}
	}
with its inequality becoming an equality iff $K$ is a closed ball.

\subsection{More about quasilinear heat dispersion}\label{s32} Below is a further two-fold observation on \eqref{e18}\&\eqref{e19}.
	\begin{itemize}
		
		\item On the one hand, as a nice complement to \eqref{e19}, we have that if $$
		\begin{cases}
		K, \overline{\Omega}\in \mathscr{K}^n;\\
		{\lambda}^\frac1{p-1}>\frac{n-p}{p-1};\\
		\mathbb B^n_r=\big\{x\in\mathbb R^n:\ |x|<r\big\};\\ \mathcal{W}_0(K)=\upsilon_n<\upsilon\le\mathcal{W}_0(\overline{\Omega}),
		\end{cases}
		$$ then there is
		\begin{equation}
		\label{e110}
		{{{\mathcal H}}}_{\Delta_p,{\lambda}}\big(K,\Omega\big)\ge {{{\mathsf{H^d}}}}_{\Delta_p,{\lambda}}\big(\overline{\mathbb B^n}, {\mathbb B}^n_r\big)\ \ \forall\ \ r\in\Bigg\{1,  \bigg(\frac{\upsilon}{\upsilon_n}\bigg)^\frac1n\Bigg\}.
		\end{equation}
		Thanks to \eqref{e18}, we are only required to check \eqref{e110}'s situation for ${{{\mathsf{H^d}}}}_{p,{\lambda}}$. However, this is equivalent to \cite[Theorem 1.1]{AC} which has been extended to \cite[Theorem 1.1]{ChW} for complete Riemannian manifolds, so the proof's detail is omitted here.

\item On the other hand, we can utilize \eqref{e110} (being comparable volumetrically to \eqref{e29}) to get that if
$$
\begin{cases}
\mathcal{L}^n(K)=\mathcal{W}_0(K)=\upsilon_n;\\
{\Omega}=K+t{\mathbb B}^n;\\
\mathcal{L}^n(\overline{\Omega})=\mathcal{W}_0(\overline{\Omega})\ge\upsilon>\upsilon_n;\\
\lambda\in \bigg\{0,  \Big(\frac{\upsilon}{\upsilon_n}\Big)^\frac1n-1\bigg\},
\end{cases}
$$
then
\begin{align*}
{{{\mathsf{H^d}}}}_{\Delta_p,{\lambda}}\big(K,K+t{\mathbb B}^n\big)=
{{{\mathsf{H^d}}}}_{p,{\lambda}}\big(K,K+t{\mathbb B}^n\big)\ge {{{\mathsf{H^d}}}}_{p,{\lambda}}\big(\overline{\mathbb B^n},(1+\lambda){\mathbb B}^n\big)={{{\mathsf{H^d}}}}_{\Delta_p,{\lambda}}\big(\overline{\mathbb B^n},(1+\lambda){\mathbb B}^n\big),
\end{align*}
which may be regarded as a resolution to \cite[Open Problem 2]{DNT} for $p=2$.
\end{itemize}

\section{Demonstration for Theorem \ref{t12}}\label{s4}

\subsection{Proof of Theorem \ref{t12}}\label{s41} This consists of two parts.
	
	\subsubsection*{Validity of (i)}\label{s41c} As a matter of fact, validating \eqref{e417e} amounts to proving that the sharp isocapacitary inequality ({\rm{IC}})
	\begin{equation}
	\label{e417}
	\left(\frac{{{\mathsf{H^l}}}_{p}(K)}{{{\mathsf{H^l}}}_{p}(\overline{\mathbb{B}^n})}\right)^\frac{1}{n-p}\ge\left(\frac{\mathcal{L}^n(K)}{\upsilon_n}\right)^\frac{1}{n}
	\ \ \forall\ \ K\in\mathscr{K}^n
	\end{equation}
	is equivalent to  the sharp isoperimetric inequality ({\rm{IP}})
	\begin{equation}
	\label{e418}
	\left(\frac{\mathcal{H}^{n-1}(\partial K)}{\sigma_{n-1}}\right)^\frac1{n-1}\ge\left(\frac{\mathcal{L}^n(K)}{\upsilon_n}\right)^\frac{1}{n}\ \ \forall\ \ K\in\mathscr{K}^n.
	\end{equation}

\begin{itemize}	
\item On the one hand,	if  \eqref{e417} is valid, then upon connecting either \eqref{e113} or \eqref{e416} to \eqref{e417}, we reach
	$$
	\left(\frac{n\big(\frac{n-p}{p-1}\big)^{p-1}\big(n^{-1}\mathcal{H}^{n-1}(\partial K)\big)^{2(p-1)}}{{\mathsf{H^l}}_p(\overline{\mathbb B^n})\big({\mathcal L}^n(K)\big)^{p-1}}
	\right)^\frac{1}{n-p}\ge	\left(\frac{{{\mathsf{H^l}}}_{p}(K)}{{{\mathsf{H^l}}}_{p}(\overline{\mathbb{B}^n})}\right)^\frac{1}{n-p}\ge\left(\frac{\mathcal{L}^n(K)}{{\upsilon_n}}\right)^\frac{1}{n},
	$$
	which is simplified to \eqref{e418}.
	
\item On the other hand, if \eqref{e418} is valid, then, without loss of generality, we may assume
	$$
	\begin{cases}
	\overline{\mathbb B^n_r}=r\overline{\mathbb B^n}\supset K\cup\overline{\mathbb B^n_\rho}=K\cup{\rho\overline{\mathbb B^n}}\ \ \forall\ \ 0<\rho\ll r<\infty;\\
	\text{$K\in\mathscr{K}_n$ with $\mathcal{L}^n(K)=\mathcal{L}^n(\overline{\mathbb B^n_\rho})$};\\
	\text{$u$ be the $p$-equilibrium potential of ${{\mathsf{H^l}}}_{p}\big(K,\mathbb B^n_r\big)$};\\
	\text{$u_\rho$ be the $p$-equilibrium potential of  ${{\mathsf{H^l}}}_{p}\big(\overline{\mathbb B^n_\rho},\mathbb B^n_r\big)$};\\
	\mathcal{L}^n(t,r)=\mathcal{L}^n\big(\{x\in \mathbb B^n_r:\ u(x)\ge t\}\big)\ \ \forall\ \ t\in [0,1];\\
	\mathcal{L}^n_\rho(t,r)=\mathcal{L}^n\big(\{x\in \mathbb B^n_r:\ u_{\rho}(x)\ge t\}\big)\ \ \forall\ \ t\in [0,1];\\
	\mathcal{H}^{n-1}(t,r)=\mathcal{H}^{n-1}\big(\{x\in \mathbb B^n_r:\ u(x)= t\}\big)\ \ \forall\ \ t\in [0,1];\\
	\mathcal{H}^{n-1}_\rho(t,r)=\mathcal{H}^{n-1}\big(\{x\in \mathbb B^n_r:\ u_{\rho}(x)= t\}\big)\ \ \forall\ \ t\in [0,1].
	\end{cases}
	$$
	Since
	$$
	[0,1]\ni t\mapsto \mathcal{L}^n(t,r)\ \ \&\ \ \mathcal{L}^n_\rho(t,r)
	$$
	are decreasing continuous functions, we have
	$$
	\begin{cases}
	\frac{d}{dt}\mathcal{L}^n(t,r)\le 0;\\
	\frac{d}{dt}\mathcal{L}^n_\rho(t,r)\le 0;\\
	\mathcal{L}^n(0,r)=\mathcal{L}^n(\overline{\mathbb B^n_r})=\mathcal{L}^n_\rho(0,r);\\
	\mathcal{L}^n(1,r)\ge\mathcal{L}^n(K)=\mathcal{L}^n(\overline{\mathbb B^n_\rho})=\mathcal{L}^n_\rho(1,r),
	\end{cases}
	$$
	thereby getting $t_\ast\in [0,1]$ such that
	$$
	0\le \left(\mathcal{L}^n(1,r)-\mathcal{L}^n_\rho(1,r)\right)-\left(\mathcal{L}^n(0,r)-\mathcal{L}^n_\rho(0,r)\right)=\frac{d}{dt}\left(\mathcal{L}^n(t,r)-\mathcal{L}^n_\rho(t,r)\right)\Bigg|_{t=t_\ast}.
	$$
	Now, by setting
	$$ t_\star=\inf\bigg\{t\in(0,1]:\mathcal{L}^n(t,r)=\mathcal{L}^n_\rho(t,r)\bigg\},$$
	we read off
	\begin{equation}
	\label{e216}
	\begin{cases}
	\left(\mathcal{L}^n(t,r)-\mathcal{L}^n_\rho(t,r)\right)\Bigg|_{t=t_\star}=0,\\
	\left(\partial_t\mathcal{L}^n(t,r)-\partial_t\mathcal{L}^n_\rho(t,r)\right)\Bigg|_{t=t_\star}\ge 0.
	\end{cases}
	\end{equation}
	Meanwhile, note that the co-area formula and H\"older's inequality imply not only
	\begin{equation}
	\label{e217}
	\big(\mathcal{H}^{n-1}(t,r)\big)^p\le{{\mathsf{H^l}}}_{p}\big(K,\mathbb B^n_r\big)\Big(-\partial_t\mathcal{L}^n(t,r)\Big)^{p-1}
	\end{equation}
	but also
	\begin{equation}
	\label{e217b}
	\mathcal{H}^{n-1}_\rho(t,r)\big)^p= {{\mathsf{H^l}}}_{p}\big(\overline{\mathbb B^n_\rho},\mathbb B^n_r\big)\Big(-\partial_t\mathcal{L}^n_\rho(t,r)\Big)^{p-1}.
	\end{equation}
	Thus, with the help of \eqref{e216}-\eqref{e217}-\eqref{e217b} \& $p\in (1,n)$, we obtain
	\begin{equation}\label{e217c}
	\left(\frac{\mathcal{H}^{n-1}\big(\{x\in \mathbb B^n_r:\ u(x)= t_\star\}\big)}{\mathcal{H}^{n-1}(\{x\in \mathbb B^n_r:\ u_\rho(x)=t_\star\}\big)}\right)^p\le \frac{{{\mathsf{H^l}}}_{p}\big(K,\mathbb B^n_r\big)}{{{\mathsf{H^l}}}_{p}\big(\overline{\mathbb B^n_\rho},\mathbb B^n_r\big)}\ \ \forall\ \ r\gg 1.
	\end{equation}
	Thanks to the basic facts that:
	$$
	\begin{cases}
	K\ \text{is a convex conductor};\\
	\{x\in \mathbb B^n_r:\ u(x)\ge t_\star\}\ \text{is a convex conductor};\\
	\{x\in \mathbb B^n_r:\ u_\rho(x)\ge t_\star\}\ \text{is a closed ball},
	\end{cases}
	$$
	we utilize \eqref{e418}, plus \eqref{e216}'s first equality, to achieve
	\begin{align*}
	\frac{\mathcal{H}^{n-1}\big(\{x\in \mathbb B^n_r:\ u(x)=t_\star\}\big)}{\mathcal{H}^{n-1}(\{x\in \mathbb B^n_r:\ u_\rho(x)= t_\star\}\big)}&=\frac{\mathcal{H}^{n-1}\big(\{x\in \mathbb B^n_r:\ u(x)= t_\star\}\big)}{\sigma_{n-1}\Big(\upsilon_n^{-1}\mathcal{L}^n(\{x\in \mathbb B^n_r:\ u_\rho(x)\ge t_\star\}\big)\Big)^\frac{n-1}{n}}\\
	&\ge\frac{\sigma_{n-1}\Big(\upsilon_n^{-1}\mathcal{L}^n\big(\{x\in \mathbb B^n_r:\ u(x)\ge t_\star\}\big)\Big)^\frac{n-1}{n}}{\sigma_{n-1}\Big(\upsilon_n^{-1}\mathcal{L}^n(\{x\in \mathbb B^n_r:\ u_\rho(x)\ge t_\star\}\big)\Big)^\frac{n-1}{n}}\\
	&=1,
	\end{align*}
	whence \eqref{e217c} yielding
	$$
	\frac{{{\mathsf{H^l}}}_{p}\big(K,\mathbb B^n_r\big)}{{{\mathsf{H^l}}}_{p}\big(\overline{\mathbb B^n_\rho},\mathbb B^n_r\big)}\ge 1\ \ \forall\ \ r\gg 1.
	$$
	Also, due to
	$$
	\begin{cases}
	\underset{{r\to\infty}}{\lim}{{\mathsf{H^l}}}_{p}\big(K,\mathbb B^n_r\big)={{\mathsf{H^l}}}_{p}(K);\\
	\underset{{r\to\infty}}{\lim}{{\mathsf{H^l}}}_{p}\big(\overline{\mathbb B^n_\rho},\mathbb B^n_r\big)={{\mathsf{H^l}}}_{p}(\overline{\mathbb B^n_\rho}),
	\end{cases}
	$$
	we obtain
	$$
	\frac{
		{{\mathsf{H^l}}}_{p}(K)}{{{\mathsf{H^l}}}_{p}(\overline{\mathbb B^n_\rho})}\ge 1,
	$$
	thereby verifying \eqref{e417} via the hypothesis
	$$
	\left(\frac{\mathcal{L}^n(K)}{\upsilon_n}\right)^\frac1n=\Bigg(
	\frac{{{\mathsf{H^l}}}_{p}(\overline{\mathbb B^n_\rho})}{{{\mathsf{H^l}}}_{p}(\overline{\mathbb B^n})}\Bigg)^\frac1{n-p}.
	$$
	
	\end{itemize}

\subsubsection*{Validity of (ii)}\label{d41a} This is split into three pieces.
\begin{itemize}

\item Firstly, for $(p,r)\in(1,n)\times(0,\infty)$ let not only $v$ be the $p$-capacitance potential of the ball $\overline{\mathbb B^n_r}$ - i.e. -
\begin{equation}
\label{e41a}
\begin{cases}
v(x)=\Big(\frac{r}{|x|}\Big)^\frac{n-p}{p-1}\ \text{as}\ \ x\in \mathbb R^n\setminus \mathbb B^n_r;\\
\Delta_p v=0 \ \ \ \ \ \ \ \ \ \ \text{in}\ \mathbb R^n\setminus \overline{\mathbb B^n_r};\\
{{\mathsf{H^l}}}_{p}\big(\overline{\mathbb B^n_r}\big)=\int_{\mathbb R^n\setminus \overline{\mathbb B^n_r}}|\nabla v|^p\,d\mathcal{L}^n=
r^{n-p}\big(\frac{p-1}{n-p}\big)^{1-p}\sigma_{n-1},
\end{cases}
\end{equation}
but also
\begin{equation}\label{e41}
\begin{cases}
g(t)=|\nabla v|\big|_{v=t}=\left(\frac{n-p}{(p-1)r}\right)t^{\frac{n-1}{n-p}};\\
F(t)=\int_t^1\frac{ds}{g(s)}=r(t^{-\frac{p-1}{n-p}}-1);\\
G=F^{-1} =\text{the inverse of $F$}.
\end{cases}
\end{equation}
Then, without loss of generality, we may assume that $K\in\mathscr{K}^n$ has smooth boundary $\partial K$ and ${\rm d}_K(x)=\text{dist}(x,K)$ is the distance from $x$ to $K.$ We consider the test function
$$u(x)=G\big({\rm d}_K(x)\big)\  \ \text{subject to}\  \
\begin{cases}
u\big|_{\partial K}=1;\\
\underset{|x|\to\infty}{\lim}u(x)=0;\\
|\nabla u(x)|=\big|G'\big({\rm{d}}_K(x)\big)\big|={\big|F'\big(u(x)\big)\big|}^{-1}=g\big(u(x)\big),
\end{cases}
$$
whence getting
\begin{equation}\label{e42}
{{\mathsf{H^l}}}_{p}(K)\le\int_{\mathbb R^n\setminus K}|\nabla u|^p\,d\mathcal{L}^n=\int_0^1\left(\int_{\{u=t\}}|\nabla u|^{p-1}d\mathcal{H}^{n-1}\right)dt=\int_0^1\big(g(t)\big)^{p-1}\mathcal{H}^{n-1}\big(\{u=t\}\big)\,dt.
\end{equation}

Upon noticing
$$
\mathcal{H}^{n-1}\big(\{u=t\}\big)=\mathcal{H}^{n-1}\big(\{x: {\rm d}_K(x)=F(t)\}\big)
=\sum_{j=1}^n\binom{n}{j}j\mathcal{W}_j(K)\big({F(t)}\big)^{j-1},
$$
we obtain that if
$$
{\rm B}(a,b)=\int_0^1 t^{a-1}(1-t)^{b-1}\,dt=\text{the Beta function in $(a,b)\in (0,\infty)^2$}
$$
then
\begin{equation}\label{e43}
\begin{aligned}
  {{\mathsf{H^l}}}_{p}(K)&\le\int_0^1\big(g(t)\big)^{p-1}\mathcal{H}^{n-1}\big(\{u=t\}\big)\,dt\\
  &=\int_0^1 \left(\frac{n-p}{(p-1)r}t^{\frac{n-1}{n-p}}\right)^{p-1}\sum_{j=1}^n\binom{n}{j}j\mathcal{W}_j(K)\bigg(r(t^{-\frac{p-1}{n-p}}-1)\bigg)^{j-1}\,dt
  \\ &= \left(\frac{n-p}{p-1}\right)^{p-1}\sum_{j=1}^n\binom{n}{j}j\mathcal{W}_j(K) r^{j-p}\int_0^1 t^{\frac{(n-1)(p-1)}{n-p}}(t^{-\frac{p-1}{n-p}}-1)^{j-1}\,dt
  \\&=\left(\frac{n-p}{p-1}\right)^p\sum_{j=1}^n
j\binom{n}{j}{\mathcal{W}}_j(K)r^{j-p}\mathrm{B}\Bigg(n-j+\frac{n-p}{p-1},j\Bigg),
\end{aligned}
\end{equation}
with its inequality becoming an equality if $K=\overline{\mathbb B^n_r}$.

\item Secondly, due to the generic Aleksandrov-Fenhel estimation (cf. \cite[(7.66)]{Sch})
	
		\begin{align}
	\label{e44}
	0\le i<j<k\le n&\Longrightarrow	\bigg(\frac{\mathcal{W}_j(K)}{\upsilon_n}\bigg)^{k-i}\ge\bigg(\frac{\mathcal{W}_i(K)}{\upsilon_n}\bigg)^{k-j}\bigg(\frac{\mathcal{W}_k(K)}{\upsilon_n}\bigg)^{j-i}\notag\\
	&\Longrightarrow
	\mathcal{W}_k(K)\le\big(\mathcal{W}_2(K)\big)^{k-1}\big(\mathcal{W}_1(K))\big)^{2-k}\ \ \text{for}\ \ n\ge k\ge 3,
	\end{align}
with two inequalities becoming an equality as $K$ is a closed ball, we not only take
$$
r=\frac{\mathcal{W}_1(K)}{\mathcal{W}_2(K)}=\frac{n^{-1}\mathcal{H}^{n-1}(\partial K)}{\mathcal{W}_2(K)}\ \ \text{within}\ \eqref{e43}
$$
but also utilize \eqref{e44}'s second implication to obtain
\begin{align}\label{e45}
{{\mathsf{H^l}}}_{p}(K)&\le \left(\frac{n-p}{p-1}\right)^p\sum_{j=1}^n
j\binom{n}{j}\left(\frac{\big({\mathcal{W}}_2(K)\big)^{j-1}}{\big({\mathcal{W}}_1(K)\big)^{j-2}}\right)
\left(\frac{\mathrm{B}\Big(n-j+\frac{n-p}{p-1},j\Big)}{\left(\frac{{\mathcal{W}}_1(K)}{\mathcal{W}_2(K)}\right)^{p-j}}\right)\\
&=\left(\frac{n-p}{p-1}\right)^p\big({\mathcal{W}}_1(K)\big)^{2-p}\big(\mathcal{W}_2(K)\big)^{p-1}\sum_{j=1}^n
j\binom{n}{j}\mathrm{B}\Big(n-j+\frac{n-p}{p-1},j\Big),\notag
\end{align}
with its equality being an equality iff $K$ is a closed ball. However, since \eqref{e45} is valid for $\overline{\mathbb B^n}$, the basic formula
$$
\begin{cases}
{{\mathsf{H^l}}}_{p}(\overline{\mathbb B^n})=\left(\frac{p-1}{n-p}\right)^{1-p}\sigma_{n-1};\\
\mathcal{W}_1(tK)=t^{n-1}\mathcal{W}_1(K)\ \ \forall\ \ t\in (0,\infty);\\
\mathcal{W}_2(tK)=t^{n-2}\mathcal{W}_2(K)\ \ \forall\ \ t\in (0,\infty),
\end{cases}
$$
is brought into \eqref{e45} to produce not only
\begin{equation}
\label{e122ee}
\sum_{j=1}^nj\binom{n}{j}\mathrm{B}\Big(n-j+\frac{n-p}{p-1},j\Big)=\frac{n(p-1)
}{n-p},
\end{equation}
but also \eqref{e113}'s first inequality with its equality case.

\item Thirdly, upon recalling the classic Aleksandrov-Fenhel inequality (which follows from \eqref{e44}'s first implication with $j=n$)
	\begin{equation}
	\label{e47}
0\le i<j\le n-1\Longrightarrow	\left(\frac{\mathcal{W}_i(K)}{\upsilon_n}\right)^\frac{1}{n-i}\le\left(\frac{\mathcal{W}_j(K)}{\upsilon_n}\right)^\frac{1}{n-j}
\end{equation}
with equality as $K$ is a closed ball, we use \eqref{e47}'s case $\{i=1, j=n-1>1\}$ to obtain
\begin{equation}
\label{e48}
1\le k\le n-1\Longrightarrow	\left(\frac{\mathcal{W}_k(K)}{\upsilon_n}\right)^\frac{1}{n-k}\le\frac{\mathcal{W}_{n-1}(K)}{\upsilon_n}
\end{equation}
with equality if $K$ is a closed ball, thereby using not only \eqref{e48} \& its equality case but also \eqref{e122ee}, as well as taking
$$
r=\frac{\mathcal{W}_{n-1}(K)}{\mathcal{W}_n(K)}=\upsilon_n^{-1}{\mathcal{W}_{n-1}(K)}\ \ \text{within}\ \ \eqref{e43},
$$
to establish \eqref{e113}'s second inequality with its equality case - in other words - if $K^\ast$ is a closed ball - for instance - $\overline{\mathbb B^n_r}$, then
$$
\eqref{eMW}\Longrightarrow {\mathsf{H^l}}_p(K^\ast)=r^{n-p}{\mathsf{H^l}}_p\big(\overline{\mathbb B^n}\big)=\Big(\upsilon_n^{-1}\mathcal{W}_{n-1}(K)\Big)^{n-p}\left(\frac{p-1}{n-p}\right)^{1-p}\sigma_{n-1}\ge{\mathsf{H^l}}_p(K),
$$
with inequality becoming an equality iff $K$ is a closed ball. Equivalently, of all convex conductors with a given mean width, the closed ball is a unique maximizer of the quasilinear heat loss.

\end{itemize}

\subsection{More about quasilinear heat loss}\label{s42} Three groups of crucial comments on Theorem \ref{t12} are in order.
	
	\begin{itemize}

		\item Upon sending $p$ to $1$ in \eqref{e417e} \& \eqref{e113}, we find that the limiting cases of \eqref{e417} \& \eqref{e418} coincide due to the fundamental formula (cf. \cite[p.188]{LXZ})
		\begin{equation}
		\label{eLXZ}
		{{\mathsf{H^l}}}_{1}(K)=\lim_{p\to 1}{{\mathsf{H^l}}}_{p}(K)=\mathcal{H}^{n-1}(\partial K)\ \ \forall\ \ K\in\mathscr{K}^n.
		\end{equation}
		In other words, the intrinsic geometry within the capacitance-potential theory for the convex conductors is just the classic isoperimetry. This can be further seen from the following discussion.
		\begin{itemize}
		
		\item In accordance with \cite[Theorem 1.1]{Xaim},  there exists such a novel isocapacitary inequality that for any compact
		convex domain $K\subseteq\mathbb{R}^n$ with smooth boundary,
		\begin{equation}
		\label{e317b}
	 \Bigg(\frac{n(p-1)}{p(n-1)}\Bigg)\left(\frac{\Big(\frac{\mathcal{H}^{n-1}(\partial K)}{{\sigma_{n-1}}}\Big)^\frac1{n-1}}
		{\Big(\frac{{{\mathsf{H^l}}}_{p}(K)}{{\mathsf{H^l}}_p(\overline{\mathbb{B}^n})}\Big)^\frac1{n-p}}\right)^\frac{n-p}{p-1}+	\Bigg(\frac{n-p}{p(n-1)}\Bigg)\left(\frac{\Big(\frac{\mathcal{L}^{n}(K)}{{\upsilon_n}}\Big)^\frac1{n}}{\Big(\frac{\mathcal{H}^{n-1}(\partial K)}{{\sigma_{n-1}}}\Big)^\frac1{n-1}}\right)^n
		\le 1\ \ \text{under}\ \ p\in (1,\infty),
		\end{equation}
		with equality iff $K$ is a closed ball. Consequently, \eqref{e317b} produces the isocapacitary inequality:
		\begin{equation*}
		\Bigg(\frac{n(p-1)(n-p)}{p^2(n-1)^2}\Bigg)^{p-1}\left(\frac{\big(\frac{\mathcal{H}^{n-1}(\partial K)}{{\sigma_{n-1}}}\big)^\frac1{n-1}}{\big(\frac{{{\mathsf{H^l}}}_{p}(K)}
			{{{{\mathsf{H^l}}}}_p(\overline{\mathbb{B}^n})}\big)^\frac1{n-p}}\right)^{n-p}	\left(\frac{\big(\frac{\mathcal{L}^{n}(K)}{{\upsilon_n}}\big)^\frac1{n}}{\big(\frac{\mathcal{H}^{n-1}(\partial K)}{{\sigma_{n-1}}}\big)^\frac1{n-1}}\right)^{n(p-1)}\le 4^{1-p}\ \ \text{under}\ \ p\in(1,n).
		\end{equation*}
		The last inequality, along with the $n$-dimensional (IP)-\eqref{e418}, derives
		\begin{equation}
		\label{e423}
		\frac{{{\mathsf{H^l}}}_{p}(K)}{{{{\mathsf{H^l}}}}_p(\overline{\mathbb{B}^n})}\ge \Bigg(\frac{4n(p-1)(n-p)}{p^2(n-1)^2}\Bigg)^{p-1}\frac{\big(\sigma_{n-1}^{-1}\mathcal{H}^{n-1}(\partial K)\big)^p}{\big(\upsilon_n^{-1}{\mathcal{L}^{n}(K)}\big)^{p-1}}\ \ \forall\ \ (p,K)\in (1,n)\times\mathscr{K}^n,
		\end{equation}
		with its inequality approaching an equality as $p\to 1$ owing to \eqref{eLXZ}.
		
		\item Actually, \cite[Theorem 3.1(i)]{JXaim} reveals such a fundamental fact that \eqref{e417e}'s isocapacitary inequality (IC) for all compact subsets of $\mathbb R^n$ is equivalent to the following sharp Lorentz-Sobolev $\big\{ 1\le p<n,\, np(n-p)^{-1}\big\}$-inequality for $f\in W_0^{1,p}(\mathbb R^n)$:
		$$
		\|f\|_{L^{\frac{np}{n-p},\infty}(\mathbb R^n)}\le\begin{cases} n^{-\frac1p}\upsilon_n^{-\frac{1}{n}}\Big(\frac{n-p}{p-1}\Big)^\frac{1-p}{p}\|\nabla f\|_{L^p(\mathbb R^n)}&\ \ \text{as}\ \ p\in(1,n);\\
		n^{-1}\upsilon_n^{-\frac{1}{n}}\|\nabla f\|_{L^1(\mathbb R^n)}&\ \ \text{as}\ \ p=1.
		\end{cases}
		$$
	\end{itemize}

	\item A close relationship between \eqref{e113}'s right-hand-sides can be observed in the sequel.
	\begin{itemize}
	\item Upon recalling that if $\partial K$ is of $C^2$ then \eqref{e47} \& \eqref{e113} yield the following sharp estimation (cf. \cite[(2.3)]{BFNT} - \cite[Theorem 1(iii)]{BN} - \cite[Theorem 1.1]{DHMT} - \cite[Theorem 3.1]{Xaim})
	\begin{equation}
	\label{e412}
	\begin{cases}
	\mathcal{W}_2(K)=n^{-1}\int_{\partial K}\mathrm{H}_{\partial K}\,d\mathcal{H}^{n-1}
	\ge n^\frac{2-n}{n-1}\upsilon_n^\frac1{n-1}\big(\mathcal{H}^{n-1}(\partial K)\big)^\frac{n-2}{n-1};\\
	{{\mathsf{H^l}}}_{p}(K)\le \left(\frac{n-p}{p-1}\right)^{p-1}{\big(\mathcal{H}^{n-1}(\partial K)\big)^{2-p}}{\Big(\int_{\partial K}\mathrm{H}_{\partial K}\,d\mathcal{H}^{n-1}\Big)^{p-1}},
	\end{cases}
	\end{equation}
	where
	$$
	\mathrm{H}_{\partial K}=(n-1)^{-1}\sum_{j=1}^{n-1}\kappa_j=\text{the mean curvature of $\partial K$},
	$$
	Consequently, if $p\in (1,2]$ and  $\partial K$ is of $C^2$, then employing not only \eqref{e412} but also \eqref{e113}'s first inequality derives
	\begin{equation}
	\label{e413}
	{{\mathsf{H^l}}}_{p}(K)\le \Big(\frac{p-1}{n-p}\Big)^{1-p}\sigma_{n-1}\left(\frac{\mathcal{W}_2(K)}{\upsilon_n}\right)^\frac{n-p}{n-2}\ \ \text{under}\ \ n\ge 3.
	\end{equation}
	Clearly, \eqref{e413}'s case $p=2$ is just \cite[(5)]{Ber} -i.e.-
	\begin{equation}
	\label{e414}
	{{\mathsf{H^l}}}_{2}(K)\le n(n-2)\mathcal{W}_2(K)=(n-2)\int_{\partial K}\mathrm{H}_{\partial K}\,d\mathcal{H}^{n-1}.
	\end{equation}
	Thus \eqref{e113}'s two inequalities under $p=2$ are equally important in that their settings under $n=3$ are the same as \eqref{e17}'s case $n=3$.
	
\item  Tied with \eqref{e113}'s second inequality is \cite[Lemma 2.3(ii)]{Xaig}  (cf. \cite[Corollary 4.2]{HPR} for $p=2$)
showing that if
		$$
		0\le \mathrm{H}_{\partial K}\le\kappa\ \ \text{on}\ \ \partial K\ \ \text{being of $C^2$}
		$$
		then
		\begin{equation}
		\label{e415}
		\left(\Bigg(\bigg(\frac{p-1}{n-p}\bigg)^{p-1}\bigg(\frac{{{\mathsf{H^l}}}_{p}(K)}{\sigma_{n-1}}\bigg)\Bigg)^\frac1{n-p}\right)^\frac{n-p}{n-1}\le \kappa^\frac{p-1}{n-1}\bigg(\frac{\mathcal{H}^{n-1}(\partial K)}{\sigma_{n-1}}\bigg)^\frac1{n-1}\ \ \forall\ \ p\in (1,n),
		\end{equation}
		with equality iff $K$ is a closed ball of radius $\kappa^{-1}$.  Interestingly, \eqref{e113} or \eqref{e413} not only extends \eqref{e414} at large, but also can, along with
\eqref{e44}'s first inequality for $i=j-1=k-2=0$ -i.e.-
$$
\big(n^{-1}\mathcal{H}^{n-1}(\partial K)\big)^2=\big({\mathcal{W}_1(K)}\big)^2\ge\mathcal{W}_0(K){\mathcal{W}_2(K)}=\mathcal{L}^n(K){\mathcal{W}_2(K)},
$$		
generate the sharp inequality (cf. \cite[(12)]{BN} extending the linear case \cite[(6)]{Ber} for any $C^2$-convex conductor)
\begin{equation}
\label{e416}
\frac{{{\mathsf{H^l}}}_{p}(K)}{{{{\mathsf{H^l}}}}_p(\overline{\mathbb{B}^n})}\le \frac{\big(\sigma_{n-1}^{-1}\mathcal{H}^{n-1}(\partial K)\big)^p}{\big(\upsilon_n^{-1}{\mathcal{L}^n(K)}\big)^{p-1}}\ \ \forall\ \ K\in\mathscr{K}^n.
\end{equation}
Nicely, \eqref{e416} may be treated as a reversed variant of \eqref{e423}.
\end{itemize}

 \item As established within \cite[Theorem 1 (i)-(ii)]{BN}, for $K\in\mathscr{K}^{n}$ with  {Lipschitz} boundary $\partial K$ there is
	\begin{equation}
	\label{e114}
	{{\mathsf{H^l}}}_{p}(K)\le\left(\int_0^\infty \big(\mathcal{H}^{n-1}(\partial K_t)\big)^\frac1{1-p}\,dt\right)^{1-p}\ \ \text{under}\ \ p\in (1,n)\ \ \&\ \  K_t=\big\{x\in\mathbb R^n:\ {\rm{dist}}(x,K)\le t\big\},
	\end{equation}
	with equality being valid iff $K$ is a closed ball.	
\begin{itemize}
	\item Indeed, just following the argument for \cite[Theorem 1]{Ber} \& \cite[(3.7)]{Xaim}, we can use
$$
\begin{cases}
\phi=f({\rm{d}}_{K_s});\\
K_s=\big\{x\in\mathbb R^n:\ {\rm{d}}_K(x)\le s\big\},
\end{cases}
$$
to get
\begin{equation}
\label{e49}
\int_{\mathbb R^n}|\nabla\phi|^p\,d\mathcal{L}^n\le\int_0^\infty|f'(r)|^p\mathcal{H}^{n-1}(\partial K_{r+s})\,dr.
\end{equation}
Upon minimizing the last inequality over all smooth functions $f$ with $f(0)-1=0=f(\infty),$ we obtain that if $r\in (0,\infty)$ then

\begin{equation}
\label{e410}
\begin{cases}
\frac{d}{dr}\left(|f'(r)|^{p-1}\mathcal{H}^{n-1}(\partial K_{r+s})\right)=0;\\
f(r)=\frac{\int_r^\infty \left(\mathcal{H}^{n-1}(\partial K_{s+t})\right)^{\frac1{1-p}}\,dt}{\int_0^\infty \left(\mathcal{H}^{n-1}(\partial K_{s+t})\right)^{\frac1{1-p}}\,dt}.
\end{cases}
\end{equation}
A combination of \eqref{e49}-\eqref{e410} derives
$$
{{\mathsf{H^l}}}_{p}(K)\le \left(\int_0^\infty \left(\mathcal{H}^{n-1}(\partial K_{s+t})\right)^{\frac1{1-p}}\,dt\right)^{1-p}\ \ \forall\ \ s\in (0,\infty).
$$
This in turn implies \eqref{e114} with equality being valid for $K$ being a closed ball.

\item As an aside, an application of
  H\"older's inequality and \eqref{e114} as well as the basic monotonicity
 $$
 0<t\le s<\infty\Longrightarrow\mathcal{H}^{n-1}(\partial K_t)\le\mathcal{H}^{n-1}(\partial K_s)
 $$
 gives
 \begin{align*}
 s&\le\left(\int_0^s\big(\mathcal{H}^{n-1}(\partial K_t)\big)^\frac1{1-p}\,dt\right)^\frac{p-1}{p} \left(\int_0^s\mathcal{H}^{n-1}(\partial K_t)\,dt\right)^\frac{1}{p}\\
 &\le\left(\int_0^\infty\big(\mathcal{H}^{n-1}(\partial K_t)\big)^\frac1{1-p}\,dt\right)^\frac{p-1}{p}\big(\mathcal{L}^n(K_s)\big)^\frac1p\\
 &\le\big({{\mathsf{H^l}}}_{p}(K)\big)^{-\frac{1}{p}}\big(\mathcal{L}^n(K_s)\big)^\frac1p\\
 &\le\big({{\mathsf{H^l}}}_{p}(K)\big)^{-\frac{1}{p}}\left(
 s{\mathcal{H}^{n-1}(\partial K_s)}
 \right)^\frac1p
 \ \ \forall\ \ (p,s)\in(1,n)\times (0,\infty),
 \end{align*}
 whence
$$
{{\mathsf{H^l}}}_{p}(K)\le\underset{s\in (0,\infty)}{\inf}s^{-p}\mathcal{L}^n(K_s)\le\underset{s\in (0,\infty)}{\inf}s^{1-p}{\mathcal{H}^{n-1}(\partial K_s)}\ \ \forall\ \ p\in (1,n),
$$
that is \cite[Theorem 1(iii)]{BN} whose case $p=2$ goes to \cite[(4)]{Ber}.
 \end{itemize}
\end{itemize}

\appendix

\section{Towards quermassintegrals}\label{s2}
\setcounter{equation}{0}

The celebrated Alexandrov-Fenchel quermassintegral radius inequality reads (cf. \cite{Alex1, Alex2})
\begin{equation}
\label{e21}
\left(\frac{\mathcal{W}_j(K)}{\mathcal{W}_j(\overline{\mathbb B^n})}\right)^\frac1{n-j}\ge\left(\frac{\mathcal{W}_i(K)}{\mathcal{W}_i(\overline{\mathbb B^n})}\right)^\frac1{n-i}\ \ \hbox{for}\ \ 0\le i<j\le n-1\ \ \&\ \ K\in\mathscr{K}^n,
\end{equation}
where \eqref{e21}'s equality holds iff $K$ is a closed ball. Of particular interest is that \eqref{e21} can be interpreted in either that amongst all elements in $\mathscr{K}^n$ with same volume the Euclidean balls have minimum quermassintegral or that amongst all elements in $\mathscr{K}^n$ with same quermassintegral the closed balls have maximum volume. Importantly, note that \eqref{e21} can be prolonged to the following inequality chain for $1\le j\le n-2$:
\begin{equation}\label{e22}
\begin{cases}
\int_{\mathbb S^{n-1}}{\it{h}}_K\,\frac{d\mathcal{H}^{n-1}}{\sigma_{n-1}}=\frac{\mathcal{W}_{n-1}(K)}{\mathcal{W}_{n-1}(\overline{\mathbb B^n})}\ge\left(\frac{\mathcal{W}_j(K)}{\mathcal{W}_j(\overline{\mathbb B^n})}\right)^\frac1{n-j}\ge\left(\frac{\mathcal{W}_0(K)}{\mathcal{W}_0(\overline{\mathbb B^n})}\right)^\frac1n\ge\left(\frac{\int_K e^{-|x|^2}\,d\mathcal{L}^n(x)}{\|e^{-|\cdot|^2}\|_{L^1(\mathbb R^n)}}\right)^\frac1n;\\
{\it{h}}_K({{\it{u}}})=\sup_{\eta\in K}\eta\cdot{{\it{u}}}\ \ \forall\ \ {{\it{u}}}\in\mathbb R^n\ \
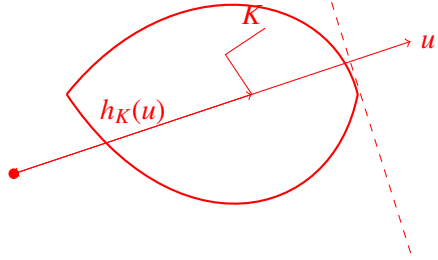
\begin{tikzpicture}[scale=3.5]
{\color{red}
% Convex body K
\draw[thick] (0.2,0.3) .. controls (0.6,0.8) and (1.2,0.7) .. (1.3,0.3)
.. controls (1.2,-0.2) and (0.6,-0.3) .. (0.2,0.3);
\node at (0.9,0.6) {$K$};

% Origin
\fill (0,0) circle (0.02);
\node[below left] at (0,0) {${}$};

% Direction u
\draw[->] (0,0) -- (1.5,0.5);
\node[right] at (1.5,0.5) {$u$};

% Supporting line
\draw[dashed] (1.5,-0.3) -- (1.09,1.0);

% h_K(u) distance
\draw[<->] (0,0) -- (0.9,0.3);
\node[above] at (0.45,0.15) {$h_K(u)$};

% Right angle marker
\draw (0.9,0.3) -- (0.8,0.45) -- (0.95,0.55);}
\end{tikzpicture}.
\end{cases}
\end{equation}

Firstly, we need two basic inequalities connecting \eqref{e22} to the pair $\{\mathcal{H}^{n-1},\mathcal{L}^n\}$.

\begin{lemma}\label{l21} Given $j\in\{0,...,n-1\}$. For $K\in\mathscr{K}^n$ let $\hbox{dia}(K)$ \& $\hbox{inr}(K)$ be its diameter \& inradius respectively. Then:
	\begin{itemize}
		\item[(i)] the diameter and quermassintegral satisfy
		\begin{equation}
		\label{e23}
		2^{-1}{\hbox{dia}(K)}\ge\int_{\mathbb S^{n-1}}{\it{h}}_K\,\frac{d\mathcal{H}^{n-1}}{\sigma_{n-1}}\ge\left(\frac{\mathcal{W}_j( K)}{\mathcal{W}_j(\overline{\mathbb B^n})}\right)^\frac{1}{n-j}.
		\end{equation}
		
		\item[(ii)] the volume, inradius and quermassintegral satisfy
		\begin{equation}\label{e24}
		\frac{\int_K e^{-|x|^2}\,d\mathcal{L}^n(x)}{\|e^{-|\cdot|^2}\|_{L^1(\mathbb R^n)}}\le\frac{\mathcal{L}^n(K)}{\upsilon_n}\le \left(\frac{\hbox{inr}(K)}{n^{-1}}\right)\left(\frac{\mathcal{W}_j(K)}{\mathcal{W}_j(\overline{\mathbb B^n})}\right)^\frac{n-1}{n-j}.
		\end{equation}
	\end{itemize}
\end{lemma}
\begin{proof} (i) The left inequality of \eqref{e23} follows from \cite[Proposition \& p.493]{Lut}. The case $j=1$ at the right inequality of \eqref{e23} is essentially the same as the case $p=n-1$ of \cite[(25)]{Cha} - however - below is a different argument. Note that not only
	\begin{equation}\label{e25}
	\int_{\mathbb S^{n-1}}{\it{h}}_K(|x|\theta)\,{d\mathcal{H}^{n-1}(\theta)}=|x|\int_{\mathbb S^{n-1}}{\it{h}}_K\,{d\mathcal{H}^{n-1}}\quad\forall\quad x\in\mathbb R^n
	\end{equation}
	but also an extension of \cite[Example 7.4]{Bor} (generated by the hitting probability of a killed Brownian motion) to $K\in\mathscr K^n$ indicates that the left side of (\ref{e25}) can be approximated by
	$\sum_{i=1}^l {\it{h}}_K(|x|\theta_i)c_i$
	that is the support function of
	$\sum_{i=1}^l c_i R_i(K),
	$
	where
	$$
	\begin{cases}
	0<c_1,...,c_l<1;\\
	\sum_{i=1}^l c_i=1;\\
	\text{$R_i(K)$ is an appropriate rotation of $K$ associated to $\theta_i$}.
	\end{cases}$$
	So, from the well-known Brunn-Minkowski inequality for $W_j(\cdot)$ (cf. \cite[(74)]{Gar} \& \cite[Corollary 9.1.5]{Sch}) it follows that
	\begin{align*}\label{23}
	\big(\mathcal{W}_j(K)\big)^\frac1{n-j}=\sum_{i=1}^l c_i\big(\mathcal{W}_j(K)\big)^\frac1{n-j}=
	\sum_{i=1}^lc_i\Big(\mathcal{W}_j\big(R_i(K)\big)\Big)^\frac1{n-j}\le\left(\mathcal{W}_j\Big(\sum_{i=1}^l c_i R_i(K)\Big)\right)^\frac1{n-j},
	\end{align*}
	where the rotation-invariance of $W_j(\cdot)$ has been used. Notice also that the right side of (\ref{e25}) is the support function of a ball with radius
	$$\int_{\mathbb S^{n-1}}{\it{h}}_K(\cdot)\,\frac{d\mathcal{H}^{n-1}(\cdot)}{\sigma_{n-1}}.
	$$
	Thus, a combination of the above approximation, the continuity of $W_j(\cdot)$, the correspondence between a support function and a convex set, the last estimation and the homogeneity (cf. Lemma \ref{l22}(i) below)
	$$\mathcal{W}_j(r\overline{\mathbb B^n})=r^{n-j}\mathcal{W}_j(\overline{\mathbb B^n}),
	$$
	yields the right inequality of (\ref{e23}).

	(ii) \eqref{e24}'s left inequality follows from
	$$
	e^{-|x|^2}\le 1\le \big(\upsilon_n\big)^{-1}\int_{\mathbb R^n}e^{-|\cdot|^2}\,d\mathcal{L}^n(\cdot)=\Gamma\Big(1+\frac{n}{2}\Big).
	$$
	However, \eqref{e24}'s right inequality follows from not only the Osserman inradius inequality in \cite{Os, Sa}
	$$
	\frac{\mathcal{L}^n(K)}{\upsilon_n}\le \left(\frac{\hbox{inr}(K)}{n^{-1}}\right)\left(\frac{\mathcal{W}_1(K)}{\mathcal{W}_1(\overline{\mathbb B^n})}\right)
	$$
	but also \eqref{e21}'s special case
	$$
	\left(\frac{\mathcal{W}_1(K)}{\mathcal{W}_1(\overline{\mathbb B^n})}\right)^\frac1{n-1}\le\left(\frac{\mathcal{W}_j(K)}{\mathcal{W}_j(\overline{\mathbb B^n})}\right)^\frac1{n-j}\quad\forall\quad j\in\{1,...,n-1\}.
	$$
\end{proof}

Secondly, given (cf. \cite[p.117, (2.43) \& 5.3.1]{Sch})
$$
\begin{cases}
K\in\mathscr{K}^n;\\
j\in\{0,...,n-1\};\\
{\it s}_{n-1-j}(\partial K,\cdot)=\frac{\varsigma_{n-1-j}\big(\kappa_1^{-1},...,\kappa^{-1}_{n-1};\partial K\big)}{\binom{n-1}{n-1-j}};\\
d\mathcal{S}_{n-1-j}(K,\cdot)={\it s}_{n-1-j}(K,\cdot)\,d\mathcal{H}^{n-1}(\cdot),
\end{cases}
$$
we have
$$
\begin{cases}
\mathcal{S}_{n-1-j}(K,B)=\int_{B}{\it s}_{n-1-j}(\partial K,\cdot)\,d\mathcal{H}^{n-1}(\cdot)\ \ \forall\ \ \hbox{Borel\ set}\ \ B\subseteq\mathbb S^{n-1};\\
\mathcal{W}_j(K)=n^{-1}S_{n-j}(K,\mathbb S^{n-1})=n^{-1}\int_{\mathbb S^{n-1}}{\it{h}}_{K}(\cdot)\,d\mathcal{S}_{n-1-j}(K,\cdot),
\end{cases}
$$
where $d\mathcal{S}_{n-1-j}(K,\cdot)$ exists actually as the pull-back measure of the so-called curvature measure  $d{\mathsf{H^l}}_{n-1-j}(K,\cdot)$ on $\partial K$ of order $n-1-j$ (cf. \cite[p.273,Theorem 4.5.6]{Sch} \& \cite{CH}):
$$
{\mathsf{H^l}}_{n-1-j}(K,E)=\int_{\partial K\cap E}\left(\frac{\varsigma_{j}\big(\kappa_1,...,\kappa_{n-1};\partial K\big)}{\binom{n-1}{j}}\right)\,d\mathcal{H}^{n-1}\quad\forall\ \ \hbox{Borel\ set}\ \ E\subseteq\mathbb R^n.
$$
Consequently, we may recall the mixed quermassintegral and its straightforward by-products.

\begin{lemma}\label{l22} Given $j\in\{0,1,...n-1\}$. For $K_{\blacktriangle},K_{\blacktriangledown}\in\mathscr{K}^n$ let
	$$
	\mathcal{W}_j(K_{\blacktriangle},K_{\blacktriangledown})=n^{-1}\int_{\mathbb S^{n-1}} {\it{h}}_{K_{\blacktriangledown}}(\xi)\,d\mathcal{S}_{n-1-j}(K_{\blacktriangle},\xi)
	$$
	be the $j$-th mixed quermassintegral of the pair $\{K_{\blacktriangle},K_{\blacktriangledown}\}$.
	Then:
	\begin{itemize}
		
		\item[(i)] the dilation-translation for quermassintegral is:
		$$
		\mathcal{W}_j(rK_{\blacktriangle}+\{x_0\})=r^{n-j}\mathcal{W}_j(K_{\blacktriangle})\ \ \text{ when}\ \ rK_{\blacktriangle}+\{x_0\}=\{rx+x_0: \ r\ge 0\}.
		$$
		
		\item[(ii)] the monotonicity  for quermassintegral is: $$
		\mathcal{W}_j(K_{\blacktriangle})\le \mathcal{W}_j(K_{\blacktriangledown})\ \ \text{when}\ \ K_{\blacktriangle}\subseteq K_{\blacktriangledown}.
		$$
		
		\item[(iii)] the Minkowski inequality for quermassintegral is:
		\begin{equation}\label{e26}
		\mathcal{W}_j(K_{\blacktriangle},K_{\blacktriangledown})\ge \big(\mathcal{W}_j(K_{\blacktriangle})\big)^\frac{n-1-j}{n-j}\big(\mathcal{W}_j(K_{\blacktriangledown})\big)^\frac1{n-j}
		\end{equation}
		with equality iff $K_{\blacktriangle},K_{\blacktriangledown}$ are homothetic.
		
		\item[(iv)] the Brunn-Minkowski inequality for quermassintegral is:
		\begin{equation*}\label{eMv}
		\big(\mathcal{W}_j(K_{\blacktriangle}+tK_{\blacktriangledown})\big)^\frac{1}{n-j}\ge \big(\mathcal{W}_j(K_{\blacktriangle})\big)^\frac{1}{n-j}+\big(\mathcal{W}_j(K_{\blacktriangledown})\big)^\frac1{n-j}
		\end{equation*}
		with equality iff $K_{\blacktriangle},K_{\blacktriangledown}$ are homothetic.
		
		\item[(v)] the Minkowski representation for quermassintegral is:
		\begin{equation*}\label{eMr}
		\mathcal{W}_j(K_{\blacktriangle})=\mathcal{W}_j(K_{\blacktriangle},K_{\blacktriangle})=n^{-1}\int_{\mathbb S^{n-1}} {\it{h}}_{K_{\blacktriangle}}(\xi)\,d\mathcal{S}_{n-1-j}(K_{\blacktriangle},\xi).
		\end{equation*}
		
		\item[(vi)] the weak convergence for quermassintegral is:
		$$
		d\mathcal{S}_{n-1-j}(K_i,\cdot)\to d\mathcal{S}_{n-1-j}(K_{\blacktriangledown},\cdot)\ \ \text{weakly as $K_i\to K_{\blacktriangledown}$ in the Hausdorff distance}
		$$
		$$
		\hbox{dist}_H(K_i,K_{\blacktriangledown})=\max\left\{\sup_{x\in K_i}\inf_{y\in K_{\blacktriangledown}}|x-y|, \sup_{x\in K_{\blacktriangledown}}\inf_{y\in K_i}|x-y|\right\}.
		$$
	\end{itemize}
\end{lemma}
\begin{proof} Needless to say, the properties (i)-(ii) are evident (cf. \cite{BCF, CG, LutGD}). The assertions (iii)-(iv) are well-known (cf. \cite[p.427, Corollary 7.6.11 \& p.406, Theorem 7.4.5]{Sch}, and imply (v) via letting $K_{\blacktriangle}=K_{\blacktriangledown}$ over there. The assertion (vi) follows from \cite[p.212, Theorem 4.2.1]{Sch}.
\end{proof}

Thirdly, we require the variational formulas for quermassintegrals. To do so, for a nonnegative continuous function $f$ on $\mathbb S^{n-1}$ -i.e.- $f\in C^+(\mathbb S^{n-1})$, let
$$
[f]=\bigcap_{\tau\in\mathbb S^{n-1}}\Big\{x\in\mathbb R^n:\ x\cdot\tau\le f(\tau)\Big\}
$$
be the so-called Alexandrov domain of $f$. Clearly,
$$
[{\it{h}}_K]=K\quad\forall\quad K\in\mathscr{K}^n.
$$

\begin{lemma}\label{l23} Given $j\in \{0,1,...,n-1\}$. For $\epsilon>0$ let
	$$
	\begin{cases} I=(-\epsilon,\epsilon);\\
	h_t(\xi): (t,\xi)\in I\times\mathbb S^{n-1}\mapsto (0,\infty)\ \text{ be continuous};\\
	\partial_t{h_t(\xi)}\Big|_{t=0}=\underset{t\to 0}{\lim} t^{-1}\big(h_t(\xi)-h_0(\xi)\big)\ \text{
		hold uniformly for $\xi\in\mathbb S^{n-1}$.}
	\end{cases}
	$$
	Then the first variational formula for quermassintegral is:
	\begin{equation}\label{e27}
	\partial_t\mathcal{W}_j\big([h_t]\big)\Big|_{t=0}=\left(1-\frac{j}{n}\right)\int_{\mathbb S^{n-1}}\partial_t{h_t(\xi)}\Big|_{t=0}\,d\mathcal{S}_{n-1-j}\big([h_0],\xi\big).
	\end{equation}
	In particular, there holds
	\begin{equation}\label{e28}
	\partial_t\mathcal{W}_j\big([{\it{h}}_K+t\phi]\big)\Big|_{t=0}=\left(1-\frac{j}{n}\right)\int_{\partial K}\phi\big(\mathsf{g}(x)\big)\,d{\mathsf{H^l}}_{n-1-j}(K,x)\ \ \forall\ \ (K,\phi)\in\mathscr{K}^n\times
	C(\mathbb S^{n-1}).
	\end{equation}
\end{lemma}
\begin{proof} Below is a two-fold argument for \eqref{e27}.
	\begin{itemize}
		\item On the one hand, the weak convergence of $$
		\text{$d\mathcal{S}_{n-1-j}([h_t],\xi)$\ {as}\ $t\to 0^+$\  (cf. Lemma \ref{l22}(vi))},
		$$
		Fatou's lemma and Lemma \ref{l22}'s \eqref{e26} are utilized to derive
		\begin{align*}
		v&\equiv \left(1-\frac{j}{n}\right)\int_{\mathbb S^{n-1}}\partial_t{h_t(\xi)}\Big|_{t=0}\,d\mathcal{S}_{n-1-j}\big([h_0],\xi\big)\\
		&\le\left(1-\frac{j}{n}\right)\liminf_{t\to 0^+}\int_{\mathbb S^{n-1}}t^{-1}\big(h_t(\xi)-h_0(\xi)\big)\,d\mathcal{S}_{n-1-j}\big([h_t],\xi\big)\\
		&=\left(1-\frac{j}{n}\right)\liminf_{t\to 0^+} t^{-1}\left(\int_{\mathbb S^{n-1}} \frac{d\mathcal{S}_{n-1-j}\big([h_t],\xi\big)}{\big(h_t(\xi)\big)^{-1}}-\int_{\mathbb S^{n-1}}\frac{d\mathcal{S}_{n-1-j}\big([h_t],\xi\big)}{\big(h_0(\xi)\big)^{-1}}\right)\\
		&=(n-j)\liminf_{t\to 0^+}t^{-1}\left(\mathcal{W}_j\big([h_t]\big)-\mathcal{W}_j\big([h_t],[h_0]\big)\right)\\
		&\le(n-j)\liminf_{t\to 0^+}t^{-1}\left(\mathcal{W}_j\big([h_t]\big)-\Big(\mathcal{W}_j\big([h_0]\big)\Big)^{\frac1{n-j}}\Big(\mathcal{W}_j\big([h_t]\big)\Big)^{\frac{n-1-j}{n-j}}\right)\\
		&=(n-j)\Big(\mathcal{W}_j\big([h_0]\big)\Big)^\frac{n-1-j}{n-j}\liminf_{t\to 0^+}t^{-1}\left(\Big(\mathcal{W}_j\big([h_t]\big)\Big)^\frac{1}{n-j}-\Big(\mathcal{W}_j\big([h_0]\big)\Big)^\frac{1}{n-j}\right).
		\end{align*}
		
		\item On the other hand, \eqref{e26}, Fatou's lemma and
		$
		{\it{h}}_{[h_t]}=h_t
		$
		are utilized to imply
		\begin{align*}
		&(n-j)\big(\mathcal{W}_j\big([h_0]\big)\big)^\frac{n-1-j}{n-j}\limsup_{t\to 0^+}t^{-1}\left(\Big(\mathcal{W}_j\big([h_t]\big)\Big)^\frac{1}{n-j}-\Big(\mathcal{W}_j\big([h_0]\big)\Big)^\frac{1}{n-j}\right)\\
		&\quad=(n-j)\limsup_{t\to 0^+}t^{-1}\left(\Big(\mathcal{W}_j\big([h_t]\big)\Big)^\frac{n-1-j}{n-j}\Big(\mathcal{W}_j\big([h_t]\big)\Big)^\frac{1}{n-j}-\mathcal{W}_j\big([h_0]\big)\right)\\
		&\quad\le(n-j)\limsup_{t\to 0^+}t^{-1}\left(\mathcal{W}_j\big([h_0],[h_t]\big)-\mathcal{W}_j\big([h_0]\big)\right)\\
		&\quad=\left(1-\frac{j}{n}\right)\limsup_{t\to 0^+}t^{-1}\left(\int_{\mathbb S^{n-1}}\frac{d\mathcal{S}_{n-1-j}([h_0],\xi)}{\big({\it{h}}_{[h_t]}(\xi)\big)^{-1}}-\int_{\mathbb S^{n-1}}\frac{d\mathcal{S}_{n-1-j}([h_0],\xi)}{{\it{h}}_{[h_0]}(\xi)}\right)\\
		&\quad\le\left(1-\frac{j}{n}\right)\int_{\mathbb S^{n-1}}\partial_t {\it{h}}_{[h_t]}(\xi)\Big|_{t=0}\,d\mathcal{S}_{n-1-j}([h_0],\xi)\\
		&\quad=\left(1-\frac{j}{n}\right)\int_{\mathbb S^{n-1}}\partial_t h_t(\xi)\Big|_{t=0}\,d\mathcal{S}_{n-1-j}([h_0],\xi)\\
		&\quad=v.
		\end{align*}
		The above estimates involving $\liminf$ and $\limsup$, along with the chain rule, give
		\begin{align*}
		v=(n-j)\mathcal{W}_j\big([h_0]\big)^\frac{n-1-j}{n-j}\partial_t\Big(\mathcal{W}_j\big([h_t]\big)\Big)^\frac1{n-j}\Big|_{t=0^+}=\partial_t\mathcal{W}_j\big([h_t]\big)\Big|_{t=0^+}.
		\end{align*}
		This last formual, together with the following formula
		$$
		\lim_{t\to 0^{-}}t^{-1}\big(h_t(\xi)-h_0(\xi)\big)=\lim_{t\to 0^{+}}(-t)^{-1}\big(h_{-t}(\xi)-h_0(\xi)\big),
		$$
		yields
		$
		v=\partial_t\mathcal{W}_j\big([h_t]\big)\Big|_{t=0^{-}},
		$
		whence reaching \eqref{e27}.
	\end{itemize}
	
	Especially, since
	$
	\partial_t\big({\it{h}}_K+t\phi\big)=\phi,
	$
	this, along with
	$$\eqref{e27}\ \ \&\ \
	\mathsf{g}_\ast\big(d{\mathsf{H^l}}_{n-1-j}(K,\cdot)\big)=d\mathcal{S}_{n-1-j}(K,\cdot),
	$$
	gives \eqref{e28} right away.
	
\end{proof}

Because
$$
\begin{cases}
\big({\sigma_{n-1}}\big)^{-1}\,d\mathcal{H}^{n-1}(\cdot)=\left({\int_{\mathbb S^{n-1}}d\mathcal{H}^{n-1}}\right)^{-1}d\mathcal{H}^{n-1}(\cdot) \ \mbox{on}\ \   \mathbb S^{n-1};\\
\Big(\|e^{-|\cdot|^2}\|_{L^1(\mathbb R^n)}\Big)^{-1}e^{-|\cdot|^2}\,d\mathcal{L}^n(\cdot)=\left(\int_{\mathbb R^n}e^{-|\cdot|^2}\,d\mathcal{L}^n(\cdot)\right)^{-1}e^{-|\cdot|^2}d\mathcal{L}^n(\cdot)\ \ \mbox{on}\ \ \mathbb R^n,
\end{cases}
$$
are two standard probability measures, it is natural to take into account an issue how to pinch the $\{0,1,...,n-1\}\ni j$-th quermassintegral radius of $K$
$$
\left(\frac{\mathcal{W}_j(K)}{\mathcal{W}_j(\overline{\mathbb B^n})}\right)^\frac1{n-j}\ \ \text{
	amongst}\ \
\mathscr{K}^n_{{w}_j}=\Big\{K\in \mathscr{K}^n:\ \mathcal{W}_j( K)\ge {\mathcal W}_j(\overline{\mathbb B^n})\Big\}.
$$
In other words, we wonder
$$
{\hbox{when\ is}\  \sup_{K\in\mathscr{K}_{w_j}^n}\Psi_{w_j,h}(K)=\sup_{K\in\mathscr{K}_{w_j}^n}{\left(\frac{\mathcal{W}_j(K)}{\mathcal{W}_j(\overline{\mathbb B^n})}\right)^\frac1{j-n}}{\left(\frac{\int_{K} h\,d\mathcal{L}^n}{\|h\|_{L^1(\mathbb R^n)}}\right)^\frac{1}{n}}
	\ \ \hbox{achievable}\ ?}
$$

In an attempt to settle the above problem (directly motivated by Yau's \cite[Problem 59]{Yau} handled within \cite[Theorem 1.1]{Xaig} (cf. \cite{XZ} for certain of the related information for ${{\mathsf{H^l}}}_p$, we recall the Gaussian map
$
\mathsf{g}: \partial K\to \mathbb S^{n-1},
$
for which $\mathsf{g}(x)$ comprises a unit normal vector for almost every $x\in\partial K$ and the induced pull-back measure of $f\,d\mathcal{H}^{n-1}$ on $\partial K$
$$
\mathsf{g}_\ast(f\,d\mathcal{H}^{n-1})(B)=\int_{\mathsf{g}^{-1}(B)}f\,d\mathcal{H}^{n-1}\quad\forall\quad\hbox{Borel\ set}\ \ B\subseteq\mathbb S^{n-1},
$$
thereby using Lemma \ref{l21}(ii) to gain the following result tied to Theorem \ref{t11}.

\begin{prop}\label{p24} Given $j\in\{0,...,n-1\}$. Let $h$ be nonnegative, continuous and Lebesgue-integrable on $\mathbb R^n$. Then the following three statements are equivalent:
	
	\begin{itemize}
		
		\item[(i)] There is $K_{\dagger}\in\mathscr{K}^n_{w_j}$ such that
		\begin{equation}
		\label{e29}
		\underset{K\in\mathscr{K}^n_{w_j}}{\sup}\Psi_{w_j,h}(K)=\Psi_{w_j,h}(K_{\dagger}).
		\end{equation}

		\item[(ii)] There is $K_{\dagger\dagger}\in\mathscr{K}^n_{w_j}$ such that
		\begin{equation}
		\label{e210}
		\mathsf{g}_\ast\left(\frac{d{\mathsf{H^l}}_{n-1-j}(K_{\dagger\dagger},\cdot)}{\int_{\mathbb S^{n-1}}{\it{h}}_{K_{\dagger\dagger}}(\cdot)\mathsf{g}_\ast\big(d{\mathsf{H^l}}_{n-1-j}(K_{\dagger\dagger},\cdot)\big)}\right)=
		\mathsf{g}_\ast\left(\frac{\big(\frac{h}{n}\big)d\mathcal{H}^{n-1}\big|_{\partial K_{\dagger\dagger}}}{\|h\|_{L^1(K_{\dagger\dagger})}}\right).
		\end{equation}
		
		\item[(iii)] There is $K_{\dagger\dagger\dagger}\in\mathscr{K}^n_{w_j}$ such that
		\begin{equation}
		\label{e211}
		0<\frac{\|h\|_{L^1(K_{\dagger\dagger\dagger})}}{\|h\|_{L^1(\mathbb R^n)}}=\frac{\int_{K_{\dagger\dagger\dagger}}h\,d\mathcal{L}^n}{\|h\|_{L^1(\mathbb R^n)}}\le 1.
		\end{equation}
		
	\end{itemize}
	And yet there is no general uniqueness up to homothety for $\big\{K_{\dagger},K_{\dagger\dagger}\big\}$ in $\big\{\eqref{e29},\eqref{e210}\big\}$.
\end{prop}

\begin{proof} The argument proceeds by the following steps.
	
	\subsubsection*{Step 1: (iii)$\Longrightarrow$(i)} Suppose that (iii) holds with \eqref{e211}. Then
	$$
	0<\Psi_{w_j,h}(K_{\dagger\dagger\dagger})=\left(\frac{\mathcal{W}_j(K_{\dagger\dagger\dagger})}{\mathcal{W}_j(\overline{\mathbb B^n})}\right)^\frac1{j-n}\left(\frac{\|h\|_{L^1(K_{\dagger\dagger\dagger})}}{\|h\|_{L^1(\mathbb R^n)}}\right)^\frac1n\le 1.
	$$
	This implies
	$$
	0<\Psi_{w_j,h}(K_{\dagger\dagger\dagger})\le\sup_{K\in\mathscr{K}^n_{w_j}}\Psi_{w_j,h}(K)\le 1.
	$$
	So, there is a sequence $\{K_i\}\subset\mathscr{K}^n_{w_j}$ such that
	$$
	{\Psi}_{w_j}(K_i)\to \underset{K\in\mathscr{K}^n_{w_j}}{\sup} {\Psi}_{w_j}(K).
	$$
	Of course, $\{K_i\}$ cannot reduce to a singleton otherwise $\mathcal{W}_j(K_i)$ reduces to zero when $i\to\infty$. Therefore, we are about to show that $\hbox{inr}(K_i)$ has a uniform positive lower bound. To do so, assume that the sequence $\big\{\hbox{inr}(K_i)\big\}$ has no a positive lower bound. Then there is a subsequence $\big\{\hbox{inr}(K_{i_k})\big\}$ which tends to zero when $k\to\infty$. If $\big\{\mathcal{W}_j(K_{i_k})\big\}$ is unbounded, then
	$$
	0<\Psi_{w_j,h}(K_{\dagger\dagger\dagger})\le\limsup_{k\to\infty}\Psi_{w_j,h}(K_{i_k})\le \limsup_{k\to\infty}\left(\frac{\mathcal{W}_j((K_{i_k})}{\mathcal{W}_j(\overline{\mathbb B^n})}\right)^\frac1{j-n}=0,
	$$
	which is a contradiction. This illustrates that $\big\{\mathcal{W}_j(K_{i_k})\big\}$ is bounded. However, the second inequality in \eqref{e24} of Lemma \ref{l21}(ii) gives
	$$
	\begin{cases}
	\frac{\mathcal{L}^n(K_{i_k})}{\upsilon_n}\le \left(\frac{\hbox{inr}(K_{i_k})}{n^{-1}}\right)\left(\frac{\mathcal{W}_j(K_{i_k})}{\mathcal{W}_j(\overline{\mathbb B^n})}\right)^\frac{n-1}{n-j}\le\left(\frac{\hbox{inr}(K_{i_k})}{n^{-1}}\right)\underset{k\in\mathbb N}{\sup} \left(\frac{\mathcal{W}_j(K_{i_k})}{\mathcal{W}_j(\overline{\mathbb B^n})}\right)^\frac{n-1}{n-j};\\
	\underset{k\to\infty}{\lim}\mathcal{L}^n(K_{i_k})=0.
	\end{cases}
	$$
	This in turn implies the following contradiction:
	$$
	0<\Psi_{w_j,h}(K_{\dagger\dagger\dagger})\le\limsup_{k\to\infty}\Psi_{w_j,h}(K_{i_k})\le\limsup_{k\to\infty}\left(\frac{\|h\|_{L^1(K_{i_k})}}{\|h\|_{L^1(\mathbb R^n)}}\right)^\frac1n=0.
	$$
	Accordingly, $\hbox{inr}(K_i)$ has a uniform positive lower bound $r_0$. Under this, an application of \eqref{e23} in Lemma \ref{l21}(i) \& \eqref{e21} gives
	$$
	\begin{cases}
	2^{-1}{\hbox{dia}(K_{i})}\ge\left(\frac{\mathcal{W}_j( K_i)}{\mathcal{W}_j(\overline{\mathbb B^n})}\right)^\frac{1}{n-j}\ge r_0;\\
	0<\Psi_{w_j,h}(K_{\dagger\dagger\dagger})\le\limsup_{i\to\infty}\Psi_{w_j,h}(K_{i})\le\limsup_{i\to\infty}\left(\frac{\mathcal{L}^n(K_i)}{\upsilon_n}\right)^{-\frac1n},
	\end{cases}
	$$
	and so that the sequence $\big\{\mathcal{L}^n(K_i)\big\}$ is bounded from above. Since $r_0$ is a positive lower bound of the sequence $\big\{\hbox{inr}(K_i)\big\}$, the sequence $\big\{\hbox{dia}(K_i)\big\}$
	is bounded from above. Utilizing the classical Blaschke selection principle we can choose such a $\{K_i\}$'s subsequence that approaches an element $K_{\dagger}\in\mathscr{K}^n$. Since not only $\mathcal{W}_j(\cdot)$ is continuous but also $h$ is continuous on $\mathbb R^n$, one has that not only ${\Psi}_{w_j}(\cdot)$ is continuous but also
	$$
	K_{\dagger}\in\mathscr{K}^n_{w_j}\ \ \text{obeys}\ \
	{\Psi}_{w_j,h}(K_{\dagger})=\underset{K\in\mathscr{K}^n_{w_j}}{\sup}{\Psi}_{w_j,h}(K).
	$$
	Thus, (i) holds with \eqref{e29}.
	
	\subsubsection*{Step 2: (i)$\Longrightarrow$(ii)} Suppose that (i) is valid with \eqref{e29}. For $f\in C^+(\mathbb S^{n-1})$ we define
	$$
	\tilde{\Psi}_{w_j,h}(f)=\Psi_{w_j,h}\big([f]\big).
	$$
	Since $K_{\dagger}\in\mathscr{K}^n_{w_j}$ is a maximizer for $\Psi_{w_j,h}$, its support function  ${\it{h}}_{K_{\dagger}}$ is a local extreme function of the functional
	$$
	C^+(\mathbb S^{n-1})\ni\phi\mapsto\tilde{\Psi}_{w_j,h}({\it{h}}_{K_{\dagger}}+t\phi)
	\ \ \text{at}\ \ t=0.
	$$
	This implies
	$$
	\partial_t\tilde{\Psi}_{w_j,h}\big({\it{h}}_{K_{\dagger}}+t\phi)\big)\Big|_{t=0}=0.
	$$
	Now, upon utilizing \eqref{e28} in Lemma \ref{l23} and the variational formula \cite[(4)]{Tso2} for $\int_K h\,d\mathcal{L}^n$ we compute
	\begin{align*}
	0&=\partial_t\tilde{\Psi}_{w_j,h}\big({\it{h}}_{K_{\dagger}}+t\phi)\big)\Big|_{t=0}\\
	&=(j-n)^{-1}\left(\frac{\mathcal{W}_j(K_{\dagger})}{\mathcal{W}_j(\overline{\mathbb B^n})}\right)^{-\frac{1+n-j}{n-j}}\partial_t\left(\frac{\mathcal{W}_j([{\it{h}}_{K_{\dagger}}+t\phi])}{\mathcal{W}_j(\overline{\mathbb B^n})}\right)\left(\frac{\|h\|_{L^1(K_{\dagger})}}{\|h\|_{L^1(\mathbb R^n)}}\right)^\frac1n\\
	&\ \ +\ n^{-1}\left(\frac{\mathcal{W}_j(K_{\dagger})}{\mathcal{W}_j(\overline{\mathbb B^n})}\right)^{\frac{1}{j-n}}\left(\frac{\|h\|_{L^1(K_{\dagger})}}{\|h\|_{L^1(\mathbb R^n)}}\right)^\frac{1-n}n\int_{\partial K_{\dagger}}\phi\big(\mathsf{g}(x)\big)\left(\frac{h\,d\mathcal{H}^{n-1}}{\|h\|_{L^1(\mathbb R^n)}}\right)\\
	&= -n^{-1}\left(\frac{\mathcal{W}_j(K_{\dagger})}{\mathcal{W}_j(\overline{\mathbb B^n})}\right)^{-\frac{1+n-j}{n-j}}\left(\frac{\int_{\partial K_{\dagger}}\phi\big(\mathsf{g}(x)\big)\,\mathsf{g}_\ast\big(d{\mathsf{H^l}}_{n-1-j}(K_{\dagger},x)\big)}{\mathcal{W}_j(\overline{\mathbb B^n})}\right)\left(\frac{\|h\|_{L^1(K_{\dagger})}}{\|h\|_{L^1(\mathbb R^n)}}\right)^\frac1n\\
	&\ \ +\ n^{-1}\left(\frac{\mathcal{W}_j(K_{\dagger})}{\mathcal{W}_j(\overline{\mathbb B^n})}\right)^{\frac{1}{j-n}}\left(\frac{\|h\|_{L^1(K_{\dagger})}}{\|h\|_{L^1(\mathbb R^n)}}\right)^\frac{1-n}n\int_{\partial K_{\dagger}}\phi\big(\mathsf{g}(x)\big)\left(\frac{h\,d\mathcal{H}^{n-1}}{\|h\|_{L^1(\mathbb R^n)}}\right)\\
	&= -n^{-1}\left(\frac{\mathcal{W}_j(K_{\dagger})}{\mathcal{W}_j(\overline{\mathbb B^n})}\right)^{-\frac{1+n-j}{n-j}}\left(\frac{\int_{\mathbb S^{n-1}}\phi\,d\mathcal{S}_{n-1-j}(K_{\dagger},\cdot)}{\mathcal{W}_j(\overline{\mathbb B^n})}\right)\left(\frac{\|h\|_{L^1(K_{\dagger})}}{\|h\|_{L^1(\mathbb R^n)}}\right)^\frac1n\\
	&\ \ +\ n^{-1}\left(\frac{\mathcal{W}_j(K_{\dagger})}{\mathcal{W}_j(\overline{\mathbb B^n})}\right)^{\frac{1}{j-n}}\left(\frac{\|h\|_{L^1(K_{\dagger})}}{\|h\|_{L^1(\mathbb R^n)}}\right)^\frac{1-n}n\int_{\mathbb S^{n-1}}\phi\,\mathsf{g}_\ast\left(\frac{h\,d\mathcal{H}^{n-1}\big|_{\partial K_{\dagger}}}{\|h\|_{L^1(\mathbb R^n)}}\right).
	\end{align*}
	This in turn yields that
	$$
	\frac{\int_{\mathbb S^{n-1}}\phi\,\mathsf{g}_\ast\big(d{\mathsf{H^l}}_{n-1-j}(K_{\dagger},\cdot)\big)}{\mathcal{W}_j(K_{\dagger})}=\frac{\int_{\mathbb S^{n-1}}\phi\,d\mathcal{S}_{n-1-j}(K_{\dagger},\cdot)}{\mathcal{W}_j(K_{\dagger})}=
	\int_{\mathbb S^{n-1}}\phi\,\mathsf{g}_\ast\left(\frac{h\,d\mathcal{H}^{n-1}\big|_{\partial K_{\dagger}}}{\|h\|_{L^1(K_{\dagger})}}\right)
	$$
	holds for not only $\phi\in C^+(\mathbb S^{n-1})$ but also $\phi\in C(\mathbb S^{n-1})$. Accordingly, the desired equation \eqref{e210} follows from choosing $K_{\dagger\dagger}=K_{\dagger}$.
	
	\subsubsection*{\it Step 3: (ii)$\Longrightarrow$(iii)} This is evident since (ii) derives
	\begin{equation*}
	\text{
		$\|h\|_{L^1(K_{\dagger\dagger\dagger})}>0$ for some $K_{\dagger\dagger\dagger}=K_{\dagger\dagger}\in\mathscr{K}^n_{w_j}$.}
	\end{equation*}

	\subsubsection*{Step 4: Non-uniqueness} Regarding the non-uniqueness up to homothety, we just consider two functions satisfying Proposition \ref{p24}(iii): $h(\cdot)=e^{-|\cdot|^2}$ in Lemma \ref{l21}(ii) or
	$$
	h(x)=\begin{cases}
	1\quad\hbox{for}\quad x\in\mathbb B^n;\\
	e^{1-|x|}\quad\hbox{for}\quad x\in\mathbb R^n\setminus\mathbb B^n;
	\end{cases}
	$$
	and see that if $K_{\blacktriangledown}=rK_{\dagger\dagger}+\{x_0\}$ then (via Lemma \ref{l22}(i))
	$$
	\begin{cases}
	\mathcal{W}_j(K_{\blacktriangledown})=r^{n-j}\mathcal{W}_j(K_{\dagger\dagger});\\
	\big(\mathcal{W}_j(K_{\blacktriangledown})\big)^{-1}d\mathcal{S}_{n-1-j}(K_{\blacktriangledown},\cdot)=r^{-1}\big(\mathcal{W}_j(K_{\dagger\dagger})\big)^{-1}d\mathcal{S}_{n-1-j}(K_{\dagger\dagger},\cdot),
	\end{cases}
	$$
	and yet, the following equation
	$$
	r^{-1}\,\mathsf{g}_\ast\left(\frac{h\,d\mathcal{H}^{n-1}\big|_{\partial K_{\dagger\dagger}}}{\|h\|_{L^1(K_{\dagger\dagger})}}\right)=\mathsf{g}_\ast\left(\frac{h\,d\mathcal{H}^{n-1}\big|_{\partial K_{\blacktriangledown}}}{\|h\|_{L^1(K_{\blacktriangledown})}}\right)
	$$
	is not always valid.
\end{proof}


\begin{thebibliography}{99}

\bibitem{AC} P. Acampora and E. Cristoeoroni, An isoperimetric result for an energy related to the $p$-capacity. Rend. Lincei Mat. Appl. 34(2023)831-844.

\bibitem{Alex1} A.D. Alexandrov, Zur Theorie der gemischten Volumina von konvexen K\"orpern, II. Neue Ungleichungen zwischen den gemischten Volumina und ihre Anwendungen, Mat. Sb. (N.S.) 2 (1937) 1205-1238 (in Russian).

\bibitem{Alex2} A.D. Alexandrov, Zur Theorie der gemischten Volumina von konvexen K\"orpern, III. Die Erweiterung zweeier Lehrsatze Minkowskis \"uber die konvexen Polyeder auf beliebige konvexe Flachen, Mat. Sb. (N.S.) 3 (1938) 27-46 (in Russian).

\bibitem{Ba} R. Barbato, Shape optimization for a nonlinear elliptic problem related
to thermal insulation, {Rend. Lincei Mat. Appl.} 35(2024)105-119.

\bibitem{Ber} M. van den Berg, On some isoperimetric inequalities for the Newtonian capacity. Commun. Contemp. Math. 27(2025)Paper No. 2450027, 20 pp.

\bibitem{BN} M. van den Berg and N. Gavitone, On functionals involving the $p$-capacity and the $q$-torsional rigidity. Calc. Var. (2025) 64:245.

\bibitem{Bor} C. Borell, {Hitting probability of killed Brownian motion: A study on geometric regularity}, {Ann. Sci. Ecole Norm. Sup\'er. Paris} 17(1984)451-467.

\bibitem{BCF} S.G. Bobkov, A. Colesanti and I. Fragal\`a, Quermassintegrals of quasi-concave functions and generalized Pr\'ekopa-Leindler inequalities. {Manuscripta Math.} 143(2014)131-169.

 \bibitem{BMP} H. Brezis, M. Marcus and A. Ponce, Nonlinear elliptic equations with measures revisited. In Mathematical Aspects of Nonlinear Dispersive Equations (J. Bour-
 gain, C. Kenig and S. Klainerman, eds.) { Annals of Mathematics Studies} 163, Princeton Univ. Press, Princeton, NJ 2007, pp. 55-110.

 \bibitem{BP} H. Brezis and A. Ponce, Reduced measures on the boundary. {J. Funct. Anal.} {229}(2005)95-120.

\bibitem{BFNT} D. Bucur, V. Ferone, C. Nitsch and C. Trombetti, A sharp estimate for the first Robin-Laplacian eigenvalue with negative boundary parameter. Att. Accad. Naz. Lincci Rend. Lincei Mat. Appl. 30(4)(2019)665-676.

\bibitem{Cha} G. D. Chakerian, {Isoperimetric inequalities for the mean width of a convex body}, {Geometriae Ded.} 1(1973)356-362.

 \bibitem{ChW} D. Chen and Y. Wei, Comparison results for relative Robin $p$-capacity on complete Riemannian manifolds. {J. Differential Equations} 423(2025)765-796.

\bibitem{DHMT} J. Dalphin, A. Henrot, S. Masnou and T. Takahashi, On the Minimization of total mean curvature. J. Geom. Anal. doi 10.1007/s12220-015-9646-y.

\bibitem{CG} A. Colesanti and E.S. G\'omez, Functional inequalities derived from the Brunn-Minkowski inequalities for quermassintegrals.   {J. Convex Anal.} 17(2010)35-49.

\bibitem{CH} A. Colesanti and D. Hug, Hessian measures of semi-convex functions and applications to support measures of convex conductors. {Manuscripta Math.} 101(2000)221-235.


 \bibitem{DNT} F. Della Pietra, C. Nitsch and C. Trombetti, An optimal insulation problem. {Math. Ann.} 382(2022)745-759.


\bibitem{Gar} R.J. Gardner, {The Brunn-Minkowski inequality}. {Bull. Amer. Math. Soc.} 39(2002)355-405.




\bibitem{HPR} A. Hurtado, V. Palmer and M. Ritor\'e, Comparison results for capacity. Indiana Univ. Math. J. 61(2012)539-555.

\bibitem{JXaim} X. Jin and J. Xiao, Essential $p$-capacity-volume estimates for rotationally symmetric manifolds. Adv. Math. 482(2025)110603.

\bibitem{JX1} X. Jin and J. Xiao, Heat dispersion laws in smooth compact manifolds. Bull. London Math. Soc. (in press)(2026)15pages.

 \bibitem{LiHo} X. Liu and T. Horiuchi, The equivalence among $p$-capacity, $p$-Laplace-capacities and Hausdorff measure. { Math. J. Ibaraki Univ.} {50}(2018)5-13.

\bibitem{LXZ} M. Ludwig, J. Xiao and G. Zhang, Sharp convex Lorentz-Sobolev inequalities. {Math. Ann.} {350}(2011)169-197.

\bibitem{Lut} E. Lutwak, A general Bieberbach inequality. {Math. Proc. Camb. Phil. Soc.} 78(1975)493-495.


\bibitem{LutGD} E. Lutwak, Inequalities for mean circumscribing simplices. {Geometriae Ded.} 66(1997)119-134.


\bibitem{MazBook} V. Maz'ya, {{Sobolev Spaces with Applications to Elliptic Partial Differential Equations}}. 2nd, revised and augmented edition, Springer, 2011.

\bibitem{Os} R. Osserman, {Bonnesen-stype isoperimetric inequalities}. {Amer. Math. Monthly} 86(1979)1-29.

\bibitem{Po} G. P\'olya, Estimating electrostatic capacity. {Amer. Math. Monthly} 54(1947)201-206.

\bibitem{Sa} J.R. Sangwine-Yager, {A Nonnesen-style inradius inequality in 3-space}. {Pacific J. Math.} 134(1988)173-178.


\bibitem{Sch} R. Schneider, {Convex Bodies: The Brunn-Minkowski Theory}. 2nd Expanded Ed. Cambridge Univ. Press, Cambridge, 2014.





\bibitem{Tso2} K.S. Tso, {A direct method approach for the existence of convex hypersurfaces with prescribed Gauss-Kronecker curvature}, { Math. Z.} 209(1992)339-334.


\bibitem{Xaig} J. Xiao,  A maximum problem of S.T. Yau for variational $p$-capacity. {Adv. Geom.} 17(2017)483-496.

\bibitem{Xaim} J. Xiao, $P$-capacity vs surface-area. Adv. Math. 308(2017)1318-1336.


\bibitem{XZ} J. Xiao and N. Zhang, Flux \& radii within the subconformal capacity.  {Calc. Var. Partial Differential Equations} 60, Article number: 120 (2021).

\bibitem{Yau} S.-T. Yau, {Problem section}. In: S.-T. Yau (ed.) Seminar on Differential Geometry, 669-706, {Ann. Math. Stud.} 102, Princeton University Press, Princeton, N.J., 1982.


\end{thebibliography}
\end{document}